\documentclass[leqno,twoside,11pt]{article}   
\usepackage[latin1]{inputenc}
\usepackage[T1]{fontenc}
\usepackage[english]{babel}
\usepackage{times}
\usepackage{graphicx} 
\usepackage[margin=1.8cm,top=2.8cm]{geometry}
\usepackage{multicol}
\usepackage{amsmath}
\usepackage{amsfonts}
\usepackage{amsthm}
\usepackage{algorithm}
\usepackage{algorithmic}

\newcommand{\NDs}{$2N^*$ }

\newtheorem{proposition}{Proposition}
\theoremstyle{remark}
\newtheorem{remark}{Remark}
\theoremstyle{definition}
\newtheorem{definition}{Definition}
\theoremstyle{remark}
\newtheorem{example}{Example}

\begin{document}

 \title{New third order low-storage SSP explicit Runge--Kutta methods\footnote{This work was partially supported by the 
Spanish Research Grant MTM2016-77735-C3-2-P.}}

 \author{I.~Higueras$^\dagger$~and
 T.~Rold\'an\thanks{Departamento de Estad\'istica,  Inform\'atica y Matem\'aticas, Universidad P\'ublica de Navarra,
Campus de Arrosadia, 31006 Pamplona (SPAIN).
Email: higueras@unavarra.es, teo@unavarra.es}}
\date{\today}

\maketitle

\begin{abstract} 
When a high dimension  system of ordinary differential equations is solved numerically,  the computer memory capacity may be compromised. Thus, for such systems, it is important to incorporate low memory usage  to
some other properties of the scheme.  In the context of strong stability preserving (SSP) schemes, some 
low-storage methods have been considered in the literature. In this paper we study 5-stage third order $2N^*$ low-storage SSP  explicit Runge-Kutta schemes. These are SSP schemes that can be implemented with  $2N$ memory registers, where  $N$ is the dimension of the problem, and retain 
the previous time step approximation. This last property is crucial for a variable step size implementation of the scheme.  In this paper, first we show that the optimal SSP methods cannot be implemented with \NDs memory registers.  Next, two non-optimal SSP $2N^*$ low-storage methods are constructed; although their SSP coefficients are not optimal, they achieve some other interesting properties.  Finally, we show some numerical experiments. 
\end{abstract}

\section{Introduction}
Given an initial value problem of the form  
\begin{eqnarray}
   \frac{d}{dt} y(t)&=&f(y(t))\, , \qquad t\geq t_{0}\, , \label{ode}\\
   y(t_{0})&=& y_{0}\, ,\nonumber
   \end{eqnarray}
a common class of schemes to solve it are explicit Runge-Kutta (RK) methods. 
An $s$-stage explicit RK method is defined by a strictly lower triangular $s\times
s$  matrix  
${\cal A}$  and a  vector $b\in 
\mathbb{R}^s$. If
$y_{n}$ is the 
numerical approximation of the solution $y(t)$ at $t=t_{n}$, we obtain 
$y_{n+1}$, the numerical approximation of the solution at 
$t_{n+1}=t_{n}+h$, from
\begin{eqnarray}
   Y_{i}&=&y_{n}+ h \sum_{j=1}^{i-1}a_{ij}f(Y_{j})\, ,    \qquad 1\le i \le s\,,
   \label{rk2}  \\[1.5ex] 
    y_{n+1}&=&y_{n}+ h \sum_{i=1}^s b_{i} f(Y_{i})\, , 
   \label{rk1}
       \end{eqnarray}
where the internal 
stage $Y_{i}$ approximates $y(t_{n}+ c_{i} 
h)$, 
and, as usual, $c_{i}=\sum_{j=1}^{s-1}a_{ij}$. 

A naive implementation of  a standard explicit RK  method requires $s+1$ memory
registers of length $N$,  where $N$ is the dimension of the differential
problem  \eqref{ode}. For systems    with a large number of equations, the
high dimension of the problem \eqref{ode}  compromises the computer memory
capacity and  thus it is important to incorporate low memory usage  to
some other properties of the scheme. These ideas have been developed, e.g.,
in
\cite{calvo2012,koch2011imex,kennedy2000low,ketcheson2008highly,ketcheson2010runge,van1977construction,williamson1980low},
where different low-storage RK methods
have been constructed. The most commonly used low-storage implementations
are the ones by van der Houwen \cite{van1977construction} and Williamson
\cite{williamson1980low}.

Other kinds of low-storage methods have been studied in the context of
strong stability preserving (SSP) schemes
\cite{gottlieb1998tvd,gottlieb2001strong,ketcheson2008highly,ruuth2006global,spiteri2002nco}.
These
methods were introduced in \cite{ShOs} to  ensure numerical monotonicity for
problems    whose solutions
satisfy a monotonicity property for the forward Euler method. Sometimes it is
convenient to write SSP explicit RK methods in the \mbox{Shu-Osher} form, particularly when the  sparse
structure of the Shu-Osher matrices 
allows       an efficient implementation with low cost of
memory usage. 
In this way, in \cite{ketcheson2008highly,ruuth2006global} it is proven that some
optimal SSP schemes can be implemented with  $2N$ memory registers. However, in most of
the cases, this implementation does not keep the previously computed numerical
solution and thus, if the method is implemented with variable step size, an
additional memory register is required. In \cite{ketcheson2008highly} a deep
analysis is done and low-storage methods that retain the computed approximation at
the previous time step are studied. These methods are denoted by $2N^*$ and it is found
that first and second order optimal SSP methods are $2N^*$ methods. However, for
third order schemes, only the 3 and 4 stage ones are $2N^*$~methods. On the other hand, third order
optimal SSP methods with $s=k^2$ stages, $k>2$, are just  $2N$ low-storage methods. Besides SSP properties, robust explicit RK schemes  should also have some additional stability properties.  Although the 4-stage third order optimal SSP scheme can be implemented as a   $2N^*$ scheme, this method is unique and all its additional properties are determined.

In this paper we consider 5-stage third order SSP explicit RK methods and exploit their sparse
structure in order to get   schemes that can be implemented as $2N^*$ methods.
Although their SSP coefficients are not optimal, they have some other additional
relevant properties.

The rest of the paper is organised as follows. In section \ref{sec:SSP} we give a
brief introduction to SSP RK methods. Section \ref{sec:2N} is devoted to review
 low-storage methods that can be implemented in two memory registers. The particular structure of optimal 
5-stage third order SSP methods is analysed in section \ref{sec:ssp53}. There we see that
these methods cannot be implemented in two memory registers. In section
\ref{sec:SSP532N} we   obtain numerically some new optimal SSP explicit RK methods that can be
implemented as $2N^*$ methods.  Although their SSP  coefficients
are not optimal, they have  other remarkable properties. Some
numerical experiments show the efficiency of these new schemes in section
\ref{sec:numexp}.

\section{Strong Stability Preserving Runge-Kutta methods}\label{sec:SSP}
In this section we review some known concepts on SSP RK methods that will be used in this paper. 
These methods are relevant for dissipative problems \eqref{ode}, that is, problems such that the exact solution satisfies 
a monotonicity property of the form
\begin{equation}    \|y(t)\|\leq \|y(t_{0})\| \, , \qquad 
    \hbox{for all } \, t\geq t_{0}\, , \label{diss2}
\end{equation}
where  
$\|\cdot\|:\mathbb{R}^N\to \mathbb{R}$ denotes a convex functional, e.g., a 
norm or a semi-norm. A sufficient condition for \eqref{diss2} is monotonicity under  forward Euler steps  
\begin{equation}
    \left \|\, y +  \,  h \, f (y )  \, \right \| \leq 
   \|\,  y \, \|\, , \qquad \hbox{for } h \leq \Delta t_{FE} , \label{circle2}
    \end{equation}
for all $y
    \in \mathbb{R}^N$ and a fixed $\Delta t_{FE} >0$ (see, e.g., \cite[p. 501]{Kra1} or \cite[p. 1-2]{Higueras2006} for details). 
 
As $Y_i$ approximates $y(t_n+ c_i h)$ and usually $c_i\geq 0$, for dissipative problems  it makes sense  to require  numerical monotonicity, not only for the numerical solution, but also for the internal stages, that is,  
\begin{equation} \|Y_{i}\| \leq \|y_{n}\| \, , \qquad i=1\,, \ldots, 
s\, ,  \qquad \qquad  \|y_{n+1}\|\leq \|y_{n}\|\,   , \label{stages-lin} 
    \end{equation}
for all $\, n\geq 0$,   probably under a stepsize restriction $h\leq \Delta t_{_\text{MAX}}$.  The seminal papers by Spijker \cite{spijker1983contractivity,spijker1985stepsize,spijker2008stepsize} and Kraaijevanger \cite{Kra1,Kra2} on numerical contractivity issues for RK schemes, settle a theoretical framework that is valid not only for contractivity   but also  for monotonicity.  

With a different terminology and notation, the numerical preservation of monotonicity has also been investigated in the context of hyperbolic systems of conservation laws. In this setting, for different reasons, it is critical to deal with Total Variation Diminishing (TVD) schemes, and in the pioneering papers \cite{Shu1988a,ShOs}, monotonicity issues for the Total Variation semi-norm are analysed. In these references,
    high 
order methods satisfying   (\ref{stages-lin})
when the forward Euler 
discretization of (\ref{ode}) satisfies \eqref{circle2}
are studied.    
In this context, these methods are known as  SSP methods. 

The idea in \cite{Kra2,Shu1988a,ShOs} is to construct high order schemes by means of 
convex combinations of  forward Euler steps. Thus, RK schemes \eqref{rk2}-\eqref{rk1}, that in compact form are written as
\begin{align}
Y=e\otimes y_n + (\mathbb{A} \otimes I_N ) F(Y)\,, 
\end{align}
with $Y=(Y_{1}, \ldots, Y_s, y_{n+1})^t\in \mathbb{R}^{(s+1)N}$, $F(Y)=(f(Y_1), \ldots, f(Y_s), 0)^t\in \mathbb{R}^{(s+1)N}$, and 
\begin{align}\label{eq:A}
\mathbb{A}=\left(\begin{array}{cc} {\cal A} & 0 \\ b^t & 0 \end{array}\right)\, ,  
\end{align}
can be expressed as
\begin{align}\label{eq:ssp}
Y =\alpha_r\otimes y_n +  (\Lambda_r \otimes I_N) \left(Y + \frac{h}{r} F(Y)\right)\, , 
\end{align}
where $r\in \mathbb{R}$ and 
\begin{align}\label{eq:sspgen0}
\alpha_r=(I+r \mathbb{A})^{-1} e\, , \qquad \Lambda_r= r (I+r \mathbb{A})^{-1}\mathbb{A}\, . 
\end{align}
If $\alpha_r\geq 0$ and $\Lambda_r\geq 0$, where the inequalities should be understood component-wise, then the right hand side of \eqref{eq:ssp} is a convex combination of $y_n$ and forward Euler steps. 
 The radius of absolute monotonicity, also known as Kraaijevanger's coefficient or SSP coefficient is defined by
\begin{align}\label{eq:kra}
R(\mathbb{A})=\sup\left\{ r \, | \, r=0 \hbox{ or } r>0, (I+r \mathbb{A})^{-1} \hbox{ exits, and }\alpha_r\geq 0, \Lambda_r\geq 0 \right\}\, . 
\end{align}
 If the forward Euler method satisfies condition
 (\ref{circle2}), then, from  (\ref{eq:ssp}), numerical monotonicity \eqref{stages-lin} can be proven under the step size restriction
\begin{equation*}
h\leq R(\mathbb{A})\, \Delta t_{FE}\, . \label{c}
\end{equation*}
In this paper, SSP($s$,$p$) will denote $s$-stage $p$-th order SSP schemes.
Optimal SSP($s$,$p$) methods, in the sense that their SSP coefficient is the largest possible one for a given number of stages $s$ and order $p$, are well known in the literature (see, e.g.,  \cite{GotKetShuBook}).

\begin{remark}
If $A=(a_{ij})$ and $b=(b_j)$ in \eqref{eq:A}, a necessary condition for $R(\mathbb{A})> 0$ is $a_{ij}\geq 0$, and $b_j>0$  \cite[Theorem 4.2]{Kra2}. In this paper we assume that this sign condition holds.    \hfill$\square$
\end{remark}

\subsection{Shu-Osher representations}
Expression \eqref{eq:ssp} is a particular case of Shu-Osher representations of a RK method (see, e.g.,  \cite[Section 2]{HiRep}). Given a RK method with Butcher matrix $\mathbb{A}$, a representation is given in terms of  two matrices $(\Lambda, \Gamma)$ such that the matrix $I-\Lambda$ is invertible and $\mathbb{A}=(I-\Lambda)^{-1} \Gamma$; the numerical approximation of the RK scheme is written as
\begin{align}\label{eq:sspgen}
Y= \alpha\otimes y_n +  (\Lambda \otimes I_N)  Y +  h (\Gamma \otimes I_N)  F(Y)\, , 
\end{align}
where $\alpha=(I-\Lambda) e$. It is well known that the representation of a  RK method is not unique.

For explicit RK methods, $Y_1=y_n$  and thus the elements $\alpha_i\,, i=2,\ldots,s+1$\,,  in \eqref{eq:sspgen} can be added to the first column of the matrix $\Lambda$.  In this way, we obtain an equivalent Shu-Osher representation with $\alpha=(1, 0, \ldots, 0)^t$ given by
  \begin{eqnarray}
     Y_1&=&y_{n}\, , \nonumber \\
     Y_i&=& \sum_{k=1}^{i-1} \left(\lambda_{ik} Y_k + h
     \, \gamma_{ik}  f\left(Y_K\right)\right) \, , \quad \quad 
     i=2, \ldots, s+1\, ,   \label{metodoalfas} \\
     y_{n+1}&=&Y_{ s+1 }\, , \nonumber 
     \end{eqnarray}
 where $\Lambda=(\lambda_{ij})$, with  $\sum_{k=1}^{i-1}\lambda_{ik}=1$, and $\Gamma=(\gamma_{ij})$. Below we give two definitions about Shu-Osher representations.

\begin{definition}\label{SO_can}
We say that a Shu-Osher representation $(\Lambda, \Gamma)$ of an explicit RK method  is canonical  if $\alpha=(1, 0, \ldots, 0)^t$.   \hfill $\square$
\end{definition}

Adding and subtracting the term $r (\Gamma  \otimes I_N) Y$, it is possible to write \eqref{eq:sspgen} as
\begin{align}\label{eq:sspgen2}
Y= \alpha\otimes y_n +  \left((\Lambda  - r \, \Gamma  ) \otimes I_N\right)  Y + r (\Gamma  \otimes I_N)  \left(Y + \frac{h}{r} F(Y)\right)\, . 
\end{align}
For $r={\cal R}(\mathbb{A})$, it can be proven \cite[Proposition 2.7]{HiRep}  that there exist  Shu-Osher representations $(\Lambda , \Gamma )$   such that $\mathbb{A}=(I-\Lambda )^{-1} \Gamma $ and 
\begin{align}\label{eq:con_rep}
\Lambda \geq 0\, , \quad \Gamma  \geq 0\, , \quad \alpha \geq 0\, , \quad \Lambda  - r \, \Gamma  \geq 0\, .
\end{align}
For these representations, the right hand side of equation \eqref{eq:sspgen2} is a convex combination of $y_n$, the internal stages and forward Euler steps. 
Observe that the largest value $r$ in \eqref{eq:con_rep} that satisfies $\Lambda  - r \, \Gamma  \geq 0$ is given by  
\begin{align}\label{eq:sspSOcoef}
r=\min_{ij} \frac{\lambda_{ij}}{\gamma_{ij}}\, , 
\end{align}
that agrees with the SSP coefficient of a RK method defined in the context of TVD schemes (see, e.g., \cite{Shu1988a}; see too  \cite{GotKetShuBook} and the references therein). In other words, these representations are optimal. 

\begin{definition}
Given a RK method with Butcher matrix $\mathbb{A}$,  and a Shu-Osher representation   $(\Lambda, \Gamma)$ such that $\mathbb{A}= (I-\Lambda)^{-1} \Gamma$, we   say that the representation $(\Lambda, \Gamma)$ is optimal if $r$ in \eqref{eq:sspSOcoef} is equal to $R(\mathbb{A})$.     
\end{definition}

\begin{example}\label{ex:opt_rep}
Given a RK method $\mathbb{A}$, consider $r=R(\mathbb{A})$, the vector $\alpha_r$ and the matrix $\Lambda_r$ in \eqref{eq:sspgen0},  and define $\Gamma_r:= \Lambda_r/r$.
For the Shu-Osher representation $(\Lambda_r, \Gamma_r)$ conditions \eqref{eq:con_rep} are fulfilled and thus  it is an optimal representation. Observe that, in this case,  $\Lambda_r-r \Gamma_r=0$  and thus \eqref{eq:sspgen2} is reduced to \eqref{eq:ssp}. \hfill $\square$
\end{example}

As it has been pointed out,  given a RK method $\mathbb{A}$, in general, there is not a unique optimal representation. The proof of Proposition 2.7 in \cite{HiRep} gives the required conditions to obtain optimal representations. More precisely, if $r=R(\mathbb{A})$, 
an optimal representation $(\Lambda, \Gamma)$ can be constructed by choosing a matrix $\Lambda$ such that the following inequalities hold,
\begin{subequations}
 \begin{align}
&   r (I+r \mathbb{A})^{-1}  \mathbb{A}e \leq \Lambda e \leq e\, , \label{des2}\\
&   r \mathbb{A} (I+r \mathbb{A})^{-1}  \mathbb{A}  \leq \Lambda \mathbb{A} \leq \mathbb{A}\, ,  \label{des3}  
 \end{align}
\end{subequations}
by defining  a matrix $\Gamma$ as $\Gamma:=(I-\Lambda)  \mathbb{A}$, and imposing that
 \begin{align}\label{des4}
\Lambda\geq 0 \, , \quad  \Lambda - r \Gamma\geq 0\, .
\end{align}
With this process, usually the optimal representation is not completely determined and some  additional conditions can be imposed on the coefficients.  In section \ref{subsec:ssp53} we will use this process to construct optimal low-storage representations of an explicit RK method.

\begin{remark} \label{inv}
As it is pointed out in  \cite{ketcheson2010runge}, given a  Shu-Osher representation, the RK method is invariant under the transformation (for any $t$ and $i,j>1$)
\begin{subequations}\label{invv}%
\begin{alignat}{4}%
& \gamma_{ik}\Rightarrow \gamma_{ik} +t \gamma_{jk} \, , \label{inv1}\\
& \lambda_{ik}\Rightarrow \lambda_{ik} +t \lambda_{jk}\, , \qquad  k\neq j\, ,  \label{inv2}\\
& \lambda_{ij}\Rightarrow \lambda_{ij} - t\, .  \label{inv3} 
\end{alignat}
\end{subequations}%
In section \ref{subsec:ssp53}  we will consider this invariance property. \hfill $\square$
\end{remark}

For a detailed study on numerical monotonicity and SSP methods, see, e.g.,  \cite{ferracina:1073,FeSpi_aplnum,gottlieb2009high,Higueras2006,KeMcGo,ketcheson2008highly,Kra1,spiteri2003non}. Efficient SSP RK methods have also been analysed in  
\cite{gottlieb1998tvd,gottlieb2001strong,ruuth2006global,ShOs,spiteri2002nco}; see too \cite{GotKetShuBook} and the references therein.

\subsection{Optimal SSP methods}\label{sec:optSSP}
In this section we review some well known optimal explicit RK SSP methods  (see, e.g., \cite{Kra1}). We are particularly interested in the sparse structure of the optimal canonical Shu-Osher representations. 

Optimal SSP($s$,$1$) methods have SSP coefficient  $r=s$.  
The corresponding   Butcher coefficients $({\cal A},b)$  are \[
a_{ij}=\frac{1}{s}\,, \ 1\le j<i \le s\,; \quad b_{i}=\frac{1}{s}\,,\ 1\le i \le s\,,
\]
and the optimal canonical Shu-Osher form for these schemes is
\begin{equation}\label{ssp1so} 
\Lambda= \left(
\begin{array}{ccccc}
0 & 0 & \cdots & \cdots  & 0   \\
1 & 0 & \ddots & \ddots  & \vdots   \\
0 & 1 & \ddots &  \ddots &   \vdots  \\
\vdots & \ddots & \ddots & \ddots & \vdots   \\
0 & \cdots & 0 & 1 & 0     \\
\end{array}
\right)\, ,\qquad  \Gamma=
\left(\begin{array}{ccccc}
0 & 0 & \cdots & \cdots & 0   \\
\frac{1}{s}  & 0 & \ddots & \ddots & \vdots   \\
0 &  \frac{1}{s}  & \ddots &  \ddots & \vdots   \\
\vdots & \ddots & \ddots & \ddots & \vdots   \\
0 & \cdots & 0 &  \frac{1}{s}  & 0     \\
\end{array}
\right)\, . 
\end{equation}
Observe that they have non trivial entries just in the first
subdiagonal.
Optimal SSP($s$,$2$) methods have  SSP coefficient  $r=s-1$ and their Butcher coefficients $({\cal A},b)$  are 
\[
a_{ij}=\frac{1}{s-1}\,, \quad \ 1\le j<i \le s\,; \qquad b_{i}=\frac{1}{s}\,,\quad \ 1\le i \le s\,.
\] 
The optimal canonical  Shu-Osher representation is
\begin{equation}\label{2eo3SO}
\Lambda= \left(
\begin{array}{ccccc}
0 & 0 & \cdots & \cdots & 0   \\
1 & 0 & \ddots & \ddots & \vdots   \\
0 & \ddots & \ddots &  \ddots & \vdots   \\
\vdots & \ddots & 1 & \ddots & \vdots   \\
\frac{1}{s} & \cdots & 0 & \frac{s-1}{s} & 0     \\
\end{array}
\right)
\, ,\qquad  \Gamma=
\left(\begin{array}{ccccc}
0 & 0 & \cdots & \cdots & 0   \\
\frac{1}{s-1} & 0 & \ddots & \ddots & \vdots   \\
0 & \ddots & \ddots &  \ddots & \vdots   \\
\vdots & \ddots & \frac{1}{s-1} & \ddots & \vdots   \\
0 & \cdots & 0 & \frac{1}{s} & 0     \\
\end{array}
\right)\, . 
\end{equation}
The sparse structure of these matrices is quite similar to the one in \eqref{ssp1so},
where only the first subdiagonals are nontrivial, but now the first column of  $\Lambda$ contains an element different from zero, namely
$\lambda_{s+1,s}=1/s$.

With regard to third order schemes, the optimal  SSP(3,3) method
has SSP coefficient  $r=1$. Below we show the Butcher tableau and the Shu-Osher
matrices for this method.
\begin{equation}\label{3eo3}
\hbox{ 
\tabcolsep 5pt
\begin{tabular}{c|ccc}
 $0$  &  $0$ & $0$&  $0$    \\ [0.25ex]
$1$  & $1$ & $0$ &  $0$  \\ [0.5ex]
 $\frac{1}{2}$  & $\frac{1}{4}$ & $\frac{1}{4}$  &  $0$   \\[0.5ex]
\hline
    & $\frac{1}{6}$ & $\frac{1}{6}$  & $\frac{2}{3}$ 
\end{tabular} \,,     
}  \qquad \qquad
\Lambda=
\left(
\begin{array}{cccc}
0 & 0 & 0 & 0  \\[0.25ex]
1 & 0 & 0 & 0 \\[0.25ex]
\frac{3}{4}  &  \frac{1}{4}  & 0 & 0  \\[0.25ex]
\frac{1}{3} & 0 & \frac{2}{3} & 0  
\end{array}
\right)\,,\qquad
\Gamma=\left(
\begin{array}{cccc}
0 & 0 & 0 & 0  \\[0.25ex]
1  & 0 & 0 & 0  \\[0.25ex]
0 & \frac{1}{4} & 0 & 0  \\[0.25ex]
0 & 0 & \frac{2}{3} & 0   
\end{array}
\right)\,.
\end{equation}
These are the coefficients of the optimal SSP(4,3)  method, with  SSP coefficient  $r=2$.
\begin{equation}\label{2N3ordern2}
\hbox{ 
\tabcolsep 5pt
\begin{tabular}{c|cccc}
 $0$  &  $0$ & $0$&  $0$&  $0$     \\ [0.25ex]
$\frac{1}{2}$  & $\frac{1}{2}$ & $0$ &  $0$ &  $0$ \\ [0.6ex]
$1$  & $\frac{1}{2}$ & $\frac{1}{2}$  & $0$& $0$   \\[0.6ex]
$\frac{1}{2}$  & $\frac{1}{6}$ & $\frac{1}{6}$  & $\frac{1}{6}$& $0$   \\[0.6ex]
\hline
    & $\frac{1}{6}$ & $\frac{1}{6}$  & $\frac{1}{6}$&  $\frac{1}{2}$
\end{tabular}      
}  \qquad \qquad
\Lambda=
\left(
\begin{array}{ccccc}
0 & 0 & 0 & 0 & 0 \\
1 & 0 & 0 & 0 & 0 \\
0 & 1 & 0 & 0 & 0 \\
\frac{2}{3} & 0 & \frac{1}{3} & 0 & 0 \\
0 & 0 & 0 & 1 & 0 \\
\end{array}
\right)\,,\qquad
\Gamma=\left(
\begin{array}{ccccc}
0 & 0 & 0 & 0 & 0 \\
\frac{1}{2}  & 0 & 0 & 0 & 0 \\
0 & \frac{1}{2} & 0 & 0 & 0 \\
0 & 0 & \frac{1}{6} & 0 & 0 \\
0 & 0 & 0 & \frac{1}{2} & 0 \\
\end{array}
\right) \,.
\end{equation}
Observe that the sparse structure of the Shu-Osher forms in \eqref{3eo3} and \eqref{2N3ordern2} is the same as the one in \eqref{2eo3SO}:   some elements in the first column of $\Lambda$ are different from zero and, in the rest of the columns, only the first subdiagonal element is nontrivial; for matrix  $\Gamma$  only the first  subdiagonal contains elements different from zero.

For $s=n^2$ stages, with $n>2$, optimum   third order SSP methods with SSP coefficient $r=n^2-n$ have been found in
\cite[Theorem 3]{ketcheson2008highly}. 
For these schemes, matrices $\Lambda$ and $\Gamma$ have a sparse structure: for matrix $\Lambda$, one element in the $(1+(n-1)(n-2)/2)$-th column  is different from zero and, in the rest of the columns, only the first   subdiagonal is nontrivial;  in matrix $\Gamma$ only the first subdiagonal contains elements different from zero.

As far as we know, a detailed study on the sparse  properties of optimal canonical Shu-Osher form for optimal SSP($s$,$3$) methods with $s\geq 5$ and $s\neq n^2$ has not been done.

\section{Low Storage  $2N$ and $2N^*$ methods}\label{sec:2N}

 Low-storage RK methods are very desirable 
 to solve problems where  
 memory management considerations are at least as important as stability
considerations. 
In the literature, different approaches to reduce the memory computer
usage of forward RK methods have been proposed
\cite{calvo2012,calvo2003minimum,cavaglieri2015,gottlieb1998tvd,gottlieb2001strong,kennedy2000low,ketcheson2008highly,ketcheson2010runge,ruuth2006global,spiteri2002nco,van1977construction,williamson1980low}.

A naive implementation of an explicit $s$-stage RK method requires $s+1$ memory registers. However,   more efficient
implementations are possible if some algebraic relations  on the coefficients are imposed.  
 Most of these efficient implementations are based on the ideas of Williamson
\cite{williamson1980low} and van der Houwen \cite{van1977construction}. Although in
very different way, in both cases  it is possible to implement these RK
methods in two memory registers, and they are usually called  $2N$ schemes, where  $N$ is
the dimension of the differential problem \eqref{ode}.

More recently, in the context of SSP methods, low-storage implementations have been obtained from the sparse structure of the Shu-Osher form \eqref{eq:sspgen} of optimal SSP methods \cite{gottlieb2001strong,ketcheson2008highly,ruuth2006global}. As we have seen in Section \ref{sec:optSSP}, this is the case for SSP($s$,1), SSP($s$,2), SSP(3,3), SSP(4,3) and SSP($n^2,3)$ schemes. In this combined analysis, some optimal SSP
RK methods turn out to be optimal also in terms of the storage required
for their implementation.

In some cases, the sparse structure of the Shu-Osher matrices in  \eqref{eq:sspgen} enables
a  $2N$ low-storage implementation. However, some of these low-storage schemes do not retain $y_n$, the previous time step approximation,  and they require a third memory register to save this value. Recall that, if  $y_n$ is retained during all
the step, it can be used  to  check some accuracy or stability condition (e.g., for a variable stepsize implementation)  without additional memory usage. To differentiate both low-storage schemes the following definition is given in \cite[Section 6.1.3]{GotKetShuBook}.

\begin{definition} Given a  $2N$ low-storage RK method, we say that the RK method is a $2N^*$ low-storage scheme if $y_n$,  the numerical solution of the previous step, is retained.  \hfill $\square$
\end{definition}

In this paper we consider  RK methods with a 
canonical Shu-Osher representation of the form
\begin{align}\label{forma_2N*}
\Lambda=
\left(
\begin{array}{cccccc}
0 & 0 & 0 & 0 & 0& 0 \\
1 & 0 & 0 & 0 & 0& 0 \\
\lambda_{31} & 1- \lambda_{31} & 0 & 0 & 0& 0 \\
\vdots & 0 & \ddots & 0 & 0& 0 \\
\lambda_{s1}& 0 & 0 & 1- \lambda_{s1} & 0& 0 \\
\lambda_{s+1,1}  & 0 & 0 & 0 & 1- \lambda_{s+1,1} & 0 \\
\end{array}
\right)\,,\qquad
\Gamma=\left(
\begin{array}{cccccc}
0 & 0 & 0 & 0 & 0& 0 \\
\gamma_{21} & 0 & 0 & 0 & 0 & 0\\
0 &  \gamma_{32} & 0 & 0 & 0 & 0\\
0 & 0 &  \ddots & 0 & 0& 0 \\
0 & 0 & 0 &  \gamma_{s,s-1} & 0 & 0\\
0 & 0 & 0 & 0 &  \gamma_{s+1,s}& 0\\
\end{array}
\right)
\end{align}
These methods generalize the non-zero structure of schemes 
\eqref{2eo3SO}-\eqref{2N3ordern2}
where the nonzero
coefficients  are on the first column of $\Lambda$ and the 
first subdiagonal of $\Gamma$ and  $\Lambda$.   Schemes \eqref{forma_2N*} allow the
  $2N^*$ implementation  given in Algorithm \ref{alg:1} below.

\begin{algorithm}\caption{\ $2N^*$ implementation of scheme \eqref{forma_2N*}}\label{alg:1}
\begin{algorithmic}[1]
\STATE q1 = y
\STATE q2 = q1
\STATE q1=q1+$\gamma_{21}$*h*f(q1) \label{line3}
\FOR {$i=2$ \TO $s$}
	\STATE q1= $\lambda_{i+1,1}$*q2 +(1-$\lambda_{i+1,1}$)*q1+ 
				$\gamma_{i+1,i}$*h*f(q1) \label{line5}
\ENDFOR
\STATE y=q1
\end{algorithmic}
\end{algorithm}

The second memory register, namely q2, is needed to store the numerical solution $y_n$
required at the end to get $y_{n+1}$. From \eqref{2eo3SO}-\eqref{2N3ordern2}, it can be concluded that the optimal SSP($s$,2), SSP(3,3) and SSP(4,3) schemes are $2N^*$ low-storage methods \cite{GotKetShuBook,ketcheson2008highly}.

Besides, the implementation of schemes \eqref{forma_2N*} allows us to get a closer insight on the construction of the numerical approximation, where repeated forward Euler steps and averaged evaluations are sequentially performed. Observe that  line \ref{line3} in Algorithm \ref{alg:1} is a forward Euler step. Furthermore, in line \ref{line5}, if  $\lambda_{i+1,1} =0$ for some $i$, then  a forward
Euler step is given.

In particular, as   $\lambda_{i,1}=0$ for $i=3, \ldots, s$  for optimal SSP($s$,2) methods \eqref{2eo3SO},  these schemes  consist
of $s-1$ repeated forward Euler $h/(s-1)$-steps, followed by a last averaged evaluation at
$t_n+h$ in order to obtain a second order approximation.
For optimal SSP(4,3) method \eqref{2N3ordern2}, as $\lambda_{3,1}=\lambda_{5,1}=0$, it consists  of $2$ repeated
forward Euler $h/2$-steps and an averaged evaluation of the previous stages at
$t_n+h/2$. Additionally, a final forward Euler $h/2$-step is done.

As it has been pointed out above, the canonical Shu-Osher matrix $\Lambda$ for the optimal SSP($n^2$,3) contains a nontrivial element in the $(1+(n-1)(n-2)/2)$-th column and thus they do not belong to the  $2N^*$ low-storage class \eqref{forma_2N*}; as it is proven in \cite{ketcheson2008highly}, they can be implemented in $2N$ memory
registers. 

As far as we know, a detailed study on low-storage properties of optimal $s$-stage
third order SSP methods with $s\geq 5$ and $s\neq n^2$ has not been done.  
For $s=5$,  optimal third order SSP schemes  have been found by numerical search
in \cite{ruuth2006global,spiteri2002nco}; furthermore, numerically optimal schemes 
can be constructed with the code RK--Opt \cite{RKOpt}. In the next section we study
optimal $5$-stage third order SSP methods
and analyse their low-storage properties.

\section{Optimal 5-stage third order SSP methods}\label{sec:ssp53}
 In this section we study the structure and low-storage properties of   optimal SSP(5,3) schemes. First, we deal with the Butcher tableau of these schemes trying to obtain a closed form for some coefficients. Next, we study their Shu-Osher representations and analyse how many memory registers are required for their implementation.

 \subsection{Butcher coefficients of optimal SSP(5,3) methods} 
Different optimal SSP(5,3) methods have been numerically constructed  in the literature \cite{GotKetShuBook,ruuth2006global}.  At this moment there is a package, named RK-Opt,  that can be used to obtain optimal SSP($s$, $p$) schemes \cite{ketcheson2010runge,RKOpt}. Several runs of this code for $s=5$ and $p=3$ show  that there is a family of optimal SSP(5,3) schemes. In this section, we study this family of methods aiming at obtaining some insight in its structure that allows us to prove their low-storage properties.

 From \cite[Theorem 5.2]{Kra2}, we know that the SSP coefficient $R(\mathbb{A})$ for optimal SSP(5,3) schemes is the real root of the polynomial 
\begin{equation}
x^3-5 x^2+10 x-10=0\,  . \label{SSP3eo5}
\end{equation} 
If we denote $r=R(\mathbb{A})$,  the stability function is given by
\begin{align}\label{eq:stab_funct0}
R(z)& = \delta_1 \left(1 + \frac{z}{r}\right)  + \delta_2 \left(1 + \frac{z}{r}\right)^{2} + \delta_3 \left(1 + \frac{z}{r}\right)^{5} \, , 
\end{align}
where 
$$\delta_1=\frac{1}{4} \left(r^2-6
   r+10\right)\, , \quad \delta_2=\frac{1}{3}
   \left(-r^2+5
   r-5\right)\, , \quad \delta_3=\frac{1}{12}
   \left(r^2-2 r+2\right) \, . 
$$
Reorganizing terms, and using  that $r$ is the root of the polynomial \eqref{SSP3eo5}, the stability function \eqref{eq:stab_funct0} is reduced to
\begin{align}R(z)=1+z+\frac{z^2}{2}+\frac{z^3}{6}+\frac{z^4}{12 r}+\frac{z^5}{60 r^2}\, . \label{stab_funct}
\end{align}
The coefficients of $z^4$ and $z^5$ in \eqref{stab_funct} are equal to $b^t {\cal A}^2 c$ and $b^t {\cal A}^3 c$, respectively, and thus, optimal SSP(5,3) schemes must satisfy the conditions
\begin{align}
b^t {\cal A}^2 c=\frac{1}{12r} \, , \qquad b^t {\cal A}^3 c=\frac{1}{60 r^2}\, . \label{z4z5}
\end{align}
Furthermore, optimal SSP(5,3) schemes must also satisfy the well known third order conditions
\begin{align}
&b^t e=1\, , \quad b^t c=\frac{1}{2}\, , \quad b^t c^2=\frac{1}{3}\, , \label{cod_o30}   \\
&b^t {\cal A} c = \frac{1}{6}\, .  \label{cod_o3}\end{align}

In order to go deeper on the properties of SSP(5,3) methods, we have run several times the code  RK-Opt \cite{ketcheson2010runge,RKOpt}. 
For all the schemes obtained, we have observed the following identities for the Butcher coefficients:
\begin{align}a_{21}=a_{31}=a_{32}=\frac{1}{r}\, , \quad a_{41}=a_{42}= a_{43}\, , \quad  a_{52}=a_{53}\, . \label{con1} \end{align} 
Thus, we conclude that the Butcher tableau for optimal SSP(5,3) methods has the following structure
\begin{align}
\begin{array}{c|ccccc}
 0 & 0 & 0 & 0 & 0 & 0 \\[1ex]
 c_2 & \frac{1}{r} & 0 & 0 & 0 & 0 \\[1ex]
 c_3 &  \frac{1}{r} & \frac{1}{r} & 0 & 0 &
   0 \\[1ex]
 c_4 &  a_{41} &
a_{41}  &
a_{41} & 0 & 0 \\[1ex]
c_5 & a_{51}  &
a_{52} &
a_{52}
   & a_{54}  & 0 \\[1ex] \hline
& b_{1} &
b_{2} &
b_{3} &
 b_{4}
  &
b_5
\end{array}\label{opt_o35e}
\end{align}
Furthermore, in the different runs of the code RK-Opt, we have also noticed that the coefficient $b_3$ is always the same   concluding that it only depends on $r$.  With this information and conditions \eqref{z4z5}-\eqref{cod_o3}, we will obtain $b_3$ and some other relationships between the coefficients of these schemes. 

First, we consider conditions \eqref{z4z5} and \eqref{cod_o3}. For a RK method of the form \eqref{opt_o35e} they  are equivalent to
\begin{subequations}
\begin{align}
   &   60 \,  b_{5} \,  a_{54}  \, a_{41} -1=0 \, ,   \label{z4z5_3}
\\
  & 36 \, r \, b_{5} \,  a_{54}\, a_{41}     +12  \,  b_{4} \,  a_{41}  +12 \,   b_{5} \,  a_{52} -    r =0 \, , \label{z4z5_2}
   \\
   &  18 \, r^2 \, b_{5}\,  a_{54} \, a_{41}  \,  +18 \,  r \, b_{4} \, a_{41}  +18  \, r \, b_{5} \, a_{52} +6
 \,   b_{3}-r^2=0 \, .  \label{z4z5_1}
\end{align}
\end{subequations}
From these equations, simple computations allow us to obtain 
\begin{align}
 b_3=\frac{r^2}{60}\, , \qquad b_{4}\, a_{41}   +  b_{5} \, a_{52}  = \frac{r }{30}\, . \label{eq_red0}
\end{align}

As optimal SSP methods usually have sparse Shu-Osher matrices we study the sparsity of   the optimal canonical representation $(\Lambda_r,\Gamma_r)$ for a scheme of the form \eqref{opt_o35e} with $b_3$ given by \eqref{eq_red0}. In this process,  we observe that  
the coefficients $\gamma_{53}$ and $\gamma_{64}$ in  matrix $\Gamma_r$ are close to zero  for all the optimal SSP(5,3) schemes obtained with the code RK-Opt  \cite{RKOpt}.
This means that the equalities 
\begin{align}
b_{4}= a_{54} \, b_{5} \,  r\, , \qquad
 a_{52}=a_{41} \, a_{54} \,  r \label{eq_red}
 \end{align}  
hold and, consequently, we assume   they are true for optimal SSP(5,3) schemes. 
Finally, from  \eqref{eq_red} we obtain that $a_{54} r=  b_{4}/b_{5} =  a_{52}/a_{41}$.  In this way, we get  $ b_{4} a_{41} =b_{5} a_{52} $, and together with  \eqref{eq_red0} we find that
\begin{align}b_{4}\, a_{41}  =b_{5}\, a_{52}   = \frac{r}{60}\, .  \label{eq_red2} 
 \end{align}

Summarizing, optimal SSP(5,3) schemes belong to a 5-parametric family of methods \eqref{opt_o35e},   where $b_{1}, b_{2}, b_{4}, b_{5}$ and $a_{51}$ are the free parameters, with   
\begin{align} a_{41} =\frac{r}{60\, b_{4}}\, , \quad a_{52} =\frac{r}{60 \, b_{5}} \, ,  \quad a_{54} = \frac{b_{4}  }{b_{5}\,   r}\, , \quad b_3=\frac{r^2}{60}\, . \label{aes}
\end{align}
For this 5-parametric family, we   have not imposed the three order conditions  \eqref{cod_o30} yet, that in this case are given by 
\begin{subequations}\label{cod_o300}
\begin{align}
 &  b_1 + b_2 + b_4 +  b_5 +\frac{r^2}{60}-1=0\, ,  \label{co1} \\
& a_{51}  b_{5} r+ b_{2}+ b_{4}+\frac{7 r^2}{60}-\frac{r}{2}=0\, , \label{co2} \\
 & b_5 \left( a_{51}+\frac{b_4}{b_5 r}+\frac{r}{30
  b_5}\right)^2+\frac {b_2}{r^2}+\frac{r^2}{400
   b_4}-\frac{4}{15}=0\, . \label{co3} 
\end{align}
\end{subequations}
After this analysis, the construction of optimal SSP(5,3) schemes is easier. Observe that there are at least two free parameters that can be used to improve some other relevant properties (e.g., error constants) of the method. In this paper, we will restrict the study to low-storage implementations.   
  
In the following examples we show how the free parameters can be used to obtain a specific pattern in the Butcher tableau \eqref{opt_o35e}  of optimal SSP(5,3) schemes.   
 
\begin{example}\label{Ex_SSP53_R}
In \cite{ruuth2006global} an optimal SSP(5,3) method is obtained by numerical optimization; its Butcher tableau is of the form \eqref{opt_o35e} with $a_{51}=a_{52}$ and $b_1=b_2$. The three remaining coefficients, namely $b_{2}$, $b_{4}$ and $b_{5}$, can be obtained from the order conditions  \eqref{cod_o300}. Then, by using \eqref{aes}, the scheme in \cite{ruuth2006global} is recovered. The coefficients of this method are shown in  the Appendix section (see \eqref{method_Ruuth}). \mbox{ }\hfill $\square$
\end{example}

\begin{example} \label{Ex_SSP53_H}
If we impose $b_1=a_{51}$, $b_2=a_{52}$, and $b_4=a_{54}$, then  the last row of the Butcher matrix ${\cal A}$ coincides with vector $(b_1, b_2, b_3, b_4,0)$. From the expression of $a_{54}$ in \eqref{aes}, we obtain that $b_5=1/r$ and thus $b_3=a_{52}$.  Using that $r$ is the root of polynomial \eqref{SSP3eo5}, after some computations, we obtain an optimal SSP(5,3) method of the form \eqref{opt_o35e} with
\begin{align*} 
a_{41}=\frac{r^3+20}{3 r^4} \, , \  \  b_1=a_{51}=\frac{5 \left(r^2+32 r-38\right)}{12
   r \left(r^3+20\right)}\, , \  \  b_2=b_3=a_{52}= \frac{r^2}{60} \, , \quad   b_4=a_{54}= \frac{r^5}{20 \left(r^3+20\right)} \, , \  \   b_5= \frac{1}{r} \, . 
\end{align*}
The coefficients of this scheme are given in the Appendix section (see \eqref{RKq53new2H}). \hfill $\square$
\end{example}

\begin{remark} \label{rm:5fam}
In this section, we have shown that   optimal SSP(5,3) methods belong to the  5-parametric family of methods \eqref{opt_o35e} satisfying \eqref{aes}. From now on, we will refer to this family as the 5-parametric family of methods. Observe that the order conditions \eqref{cod_o30} (namely, \eqref{cod_o300}) have not been imposed to this family.  \hfill $\square$
\end{remark} 

\subsection{Shu-Osher low-storage form of optimal SSP(5,3) methods}\label{subsec:ssp53}
In this section we study optimal canonical  Shu-Osher representations for the 5-parametric family of methods above (see Remark~\ref{rm:5fam}). Remember that optimal SSP(5,3) schemes belong to this family. Our goal is to determine the minimum number of memory registers required for implementing them.

As it has been pointed out in Example \ref{ex:opt_rep}, for $r=R(\mathbb{A})$, an optimal representation is given by $\Lambda_r$ in \eqref{eq:sspgen0} and $\Gamma_r:=\Lambda_r/r$. For the 5-parametric family of methods, the canonical form of this optimal representation is given by  
\begin{align}\label{eq:opt_can}
\tilde \Lambda=\left(
\begin{array}{cccccc}
 0 & 0 & 0 & 0 & 0 & 0 \\
 1 & 0 & 0 & 0 & 0 & 0 \\
 0 & 1 & 0 & 0 & 0 & 0 \\
\tilde \lambda_{41}  & 0 & \frac{r^2}{60 b_4 } & 0 & 0 & 0 \\
\tilde \lambda_{51} & 0 & 0 & \frac{ b_4 }{ b_5 } & 0 & 0
   \\
\tilde \lambda_{61} & \tilde \lambda_{62} & 0 & 0 &  b_5  r & 0
\end{array}
\right)\, ,  \qquad \tilde \Gamma=\left(
\begin{array}{cccccc}
 0 & 0 & 0 & 0 & 0 & 0 \\
 \frac{1}{r} & 0 & 0 & 0 & 0 & 0 \\
 0 & \frac{1}{r} & 0 & 0 & 0 & 0 \\
 0 & 0 & \frac{r}{60 b_4 } & 0 & 0 & 0 \\
 \tilde \gamma_{51} & 0 & 0 & \frac{ b_4 }{ b_5  r} & 0
   & 0 \\
\tilde \gamma_{61} &
  \tilde \gamma_{62} & 0 & 0 &  b_5  & 0
\end{array}
\right)\, , 
\end{align}
where 
\begin{subequations}
\begin{align}
\tilde \lambda_{41} & = 1-\frac{r^2}{60 b_4 }\, , \quad \tilde \lambda_{51}   =   1-\frac{ b_4 }{ b_5 } \, , \quad \tilde \lambda_{61}   = 1 - b_2  r- b_5  r +\frac{r^3}{60} \, , \quad \tilde \lambda_{62}= r \left( b_2 -\frac{r^2}{60}\right)\, , \label{eq:opt_can1}\\
  \tilde \gamma_{51}& = a_{51}-\frac{r}{60  b_5 }\, ,  
  \quad \tilde \gamma_{61}= b_1 - b_2 - r b_5 \left( a_{51}  -\frac{r}{60 b_5} \right) \, , \quad \tilde \gamma_{62}= b_2 -\frac{r^2}{60}\, . \label{eq:opt_can2}
\end{align}
\end{subequations}
Observe that this representation is not like the $2N^{*}$ low-storage form \eqref{forma_2N*}.  Remember that, for the optimal representation in Example 1,  inequalities \eqref{eq:con_rep} are satisfied. Particularly, $\tilde \lambda_{i1}\geq 0$, for $i=4,5,6$;  $\tilde \lambda_{62}\geq 0$;  $\tilde \gamma_{i1}\geq 0$, for $i=5,6$; $\tilde \gamma_{62}\geq 0$, and $\tilde \lambda_{51}- r \tilde \gamma_{51}\geq 0$. 

In order to obtain a sparse optimal Shu-Osher representation we follow the constructive proof of Proposition 3.12 in \cite{HiRep}. Thus we consider a  lower triangular matrix $\Lambda=(\lambda_{ij})$ with arbitrary coefficients, and we define $\Gamma=(I-\Lambda)  \mathbb{A}$ and \mbox{$\alpha=(I-\Lambda)e$}.  First, we impose inequality  \eqref{des3} component-wise. It turns out that for some components  the upper and lower bound is the same and thus the middle term is determined. Next, we impose conditions on $\Gamma$  to obtain a sparse matrix such that only the first  subdiagonal is nontrivial (see \eqref{forma_2N*}).  Finally, we move elements of vector $\alpha$ to the first column of  $\Lambda$ to obtain the canonical form with \mbox{$\alpha=(1, 0 , \ldots, 0)^t$}. Proceeding in this way, and using the order condition $b^t e=1$, we obtain the canonical representation
\begin{align} \label{SO_opt}
\Lambda= \left(
\begin{array}{cccccc}
 0 & 0 & 0 & 0 & 0 & 0 \\[0.25ex]
 1 & 0 & 0 & 0 & 0 & 0 \\[0.25ex]
 0 & 1 & 0 & 0 & 0 & 0 \\[0.25ex]
\lambda_{41} & 0 & \frac{r^2}{60  b_4} & 0 & 0 & 0 \\[0.25ex]
 \lambda_{51} & \lambda_{52} & 0 &
   \frac{ b_4}{ b_5} & 0 & 0 \\[0.25ex]
\lambda_{61} & \lambda_{62} &  \lambda_{63} & 0 &  b_5 r & 0
\end{array}
\right)\, ,  \qquad 
\Gamma=\left(
\begin{array}{cccccc}
 0 & 0 & 0 & 0 & 0 & 0 \\[0.25ex]
 \frac{1}{r} & 0 & 0 & 0 & 0 & 0 \\[0.25ex]
 0 & \frac{1}{r} & 0 & 0 & 0 & 0 \\[0.25ex]
 0 & 0 & \frac{r}{60  b_4} & 0 & 0 & 0 \\[0.25ex]
 0 & 0 & 0 & \frac{ b_4}{ b_5 r} & 0 & 0 \\[0.25ex]
 0 & 0 & 0 & 0 &  b_5 & 0
\end{array}
\right)\, , 
\end{align}
where   
\begin{subequations}\label{eq:lambdas}
\begin{align} 
\lambda_{41} &=  1-\frac{r^2}{60  b_4} \, , \label{eq:l1}
\\ \lambda_{51}& = 1-\frac{ {b_4}}{ {b_5}} - r \left({a_{51}}  -\frac{r}{60  {b_5}}\right) \, , \quad \lambda_{52}=r \left({a_{51}} -\frac{r}{60  {b_5}}\right)  \, ,   \label{eq:l2} \\ 
\lambda_{61}& =  1  - {b_1} r -{b_5} r +  r  {b_5} a_{51}\, , \quad \lambda_{62}  = r \left ({b_1}  - {b_2}  -r {b_5}\left(  a_{51}   -\frac{r}{60 {b_5}}\right)\right)  \, , \quad \lambda_{63}=r \left(b_2
    -\frac{r^2}{60 }\right) .  \label{eq:l3}
\end{align}
\end{subequations}
From the order conditions \eqref{co1}-\eqref{co2}, and using that $r$ is the root of the polynomial \eqref{SSP3eo5}, we obtain that $$\lambda_{61}=\frac{1}{10} \left(-r^3+5 r^2-10 r+10\right)=0\, .$$
Observe that $\Lambda\geq 0$ and $\Gamma\geq 0$ in \eqref{SO_opt} imply that $\Lambda- r \Gamma\geq 0$.  Thus,  in order to obtain a representation with optimal SSP coefficient $r$ we only require $\lambda_{41}, \lambda_{51}, \lambda_{52}, \lambda_{62}, \lambda_{63}$, $b_4$ and $b_5$ to be non negative.
Observe that $$\lambda_{41}=\tilde  \lambda_{41}\geq 0\, , \quad \lambda_{51}= \tilde \lambda_{51}- r \, \tilde \gamma_{51} \geq 0\, , \quad \lambda_{52}=r\, \tilde \gamma_{51}\geq 0\, , \quad \lambda_{62}=r \, \tilde \gamma_{61}\geq 0\, , \quad \lambda_{63}=\tilde \lambda_{62}\geq 0\, . $$
Consequently, \eqref{SO_opt}-\eqref{eq:lambdas} is an optimal canonical representation of any first order method of the 5-parametric family of schemes. In particular, it is a  canonical  representation for optimal SSP(5,3) methods.

\begin{remark} 
The optimal canonical  Shu-Osher form  for the 5-parametric family has the sparse structure \eqref{eq:opt_can}. As it has been pointed out in Remark \ref{inv}, the  Shu-Osher  representation of a RK method is invariant under the transformation \eqref{invv}. Precisely,  this transformation can be used to obtain a 
subdiagonal matrix $\Gamma$. 
For example, in order to transform the element $\tilde\gamma_{62}$ in \eqref{eq:opt_can} into a zero, we have to make   
\eqref{inv1} equal to zero, this is  $\tilde \gamma_{62} +t \tilde \gamma_{32}=0$, to get $t=-r\tilde \gamma_{62}$. With this value of $t$ we have to transform the elements $\tilde\lambda_{62}$ and $\tilde\lambda_{63}$ in $\tilde\Lambda$ according to 
 \eqref{inv2} and \eqref{inv3}, respectively. If we repeat this process for elements $\tilde\gamma_{51}$ and $\tilde\gamma_{61}$ in $\tilde\Gamma$, we finally obtain the Shu-Osher representation in \eqref{SO_opt} for the 5-parametric family of schemes.\hfill $\square$
\end{remark} 

Shu-Osher representations of the form \eqref{forma_2N*} can be implemented with $2N^*$ memory registers. This is not the case for the optimal SSP(5,3) methods as we see in the following lemma.

\begin{proposition} The  optimal SSP(5,3) methods  are not  $2N^*$ low-storage schemes of the form \eqref{forma_2N*}. 
\end{proposition}
 
 \begin{proof}  Consider a canonical Shu-Osher representation of the form \eqref{SO_opt} and assume that   $\lambda_{52}=\lambda_{62}= \lambda_{63}=0$. In this case,  
 $$a_{51}= \frac{r}{60  b_5}\, , \quad b_1=b_2= \frac{r^2}{60}\, . 
 $$
 Inserting these values into the order conditions \eqref{co1}-\eqref{co2}, we get that
 $$b_4=\frac{r}{20} \left(10 -3 r \right)\, , \quad b_5=\frac{1}{10} \left(r^2-5 r+10\right)= \frac{1}{r}\, , 
 $$
 where we have used that $r$ is the root of the polynomial \eqref{SSP3eo5}. Now, the order condition $b^t c^2=1/2$, given in this case by \eqref{co3}, is reduced to
\begin{equation}\label{eq:pol4}
    3 r^4-40 r^3+175 r^2-330 r+250= 0\, .
\end{equation}
But  \eqref{eq:pol4} is different from zero for $r=R(\mathbb{A})$ and, 
consequently, the order condition cannot be fulfilled. 
\end{proof}
 
 Nevertheless, some optimal SSP(5,3) methods can be implemented in  $3N$  memory registers, as we see in the following proposition.
 \begin{proposition}\label{prop_3N}
 Consider a Shu-Osher representation of the form \eqref{SO_opt}. If 
the coefficients $\lambda_{52}=\lambda_{62}=0$, or the coefficient
 $\lambda_{63}=0$, 
 then the scheme can be implemented in  $3N$  memory registers.   
 \end{proposition}
 
\begin{proof}  
Below we show the corresponding $3N$ implementations for the case $\lambda_{52}=\lambda_{62}=0$ (left) and 
$\lambda_{63}=0$ (right).

\begin{minipage}[t]{0.45\textwidth}
 \begin{algorithm}[H]\caption{ \ case $\lambda_{52}=\lambda_{62}=0$}\label{alg:2}
\begin{algorithmic}[1]
\STATE q1 = y
\STATE q2 = q1
\FOR {$i=1$ \TO $2$}
	\STATE q1= q1+h*f(q1)/r
\ENDFOR
\STATE q3=q1
\FOR {$i=3$ \TO $4$}
	\STATE q1= $\lambda_{i+1,1}$*q2+$\lambda_{i+1,i}$*q1+$\gamma_{i+1,i}$*h*f(q1)
\ENDFOR
\STATE q1=$\lambda_{61}$*q2+$\lambda_{63}$*q3+$\lambda_{65}$*q1+$\gamma_{65}$*h*f(q1)
\STATE y=q1
\end{algorithmic}
\end{algorithm}
 \end{minipage}
 \qquad
  \begin{minipage}[t]{0.45\textwidth}
 \begin{algorithm}[H]\caption{ \ case  $\lambda_{63}=0$}\label{alg:3}
\begin{algorithmic}[1]
\STATE q1 = y
\STATE q2 = q1
\STATE q1=q1+h*f(q1)/r
\STATE q3=q1
\STATE q1=q1+h*f(q1)/r
\STATE q1= $\lambda_{41}$*q2+$\lambda_{43}$*q1+$\gamma_{43}$*h*f(q1) 
\FOR {$i=4$ \TO $5$}
	\STATE q1=  $\lambda_{i+1,1}$*q2+$\lambda_{i+1,2}$*q3+$\gamma_{i+1,i}$*h*f(q1)
\ENDFOR
\STATE y=q1
\end{algorithmic}
\end{algorithm}
 \end{minipage}

\bigskip

When  $\lambda_{52}=\lambda_{62}=0$ (left), a third register q3 is needed to store the third stage required to compute $y_{n+1}$. However, when $\lambda_{63}=0$ (right), the third register is needed to store the second stage we need at the end of the step. 
\end{proof}

Some optimal SSP(5,3) methods considered in  this paper can be implemented in  $3N$  memory registers:
\begin{itemize}
\item  Scheme \eqref{method_Ruuth}, constructed in \cite{ruuth2006global},  has $\lambda_{52}=\lambda_{62}=0$ and $\lambda_{63}\neq 0$;  
\item Method \eqref{RKq53new2H}, constructed in this paper,  has $\lambda_{63}=0$;
\item Method \eqref{David1}, constructed with the code RK-Opt, has $\lambda_{52}\neq 0$, $\lambda_{62}\neq 0$ and $\lambda_{63}= 0$.
\end{itemize}
However, the optimal  SSP(5,3) scheme \eqref{David2}, constructed with the code RK-Opt, has $\lambda_{52}\neq 0$, $\lambda_{62}= 0$ and $\lambda_{63}\neq 0$. Consequently, more than $3N$ registers are required to be implemented. 

 \section{Five-stage third-order  $2N^*$ Explicit Runge-Kutta methods}\label{sec:SSP532N}
In this section we consider 5-stage schemes with sparse Shu-Osher form \eqref{forma_2N*}
\begin{align}\Lambda=
\left(
\begin{array}{cccccc}
 0 & 0 & 0 & 0 & 0& 0 \\
 1 & 0 & 0 & 0 & 0& 0 \\
 \lambda_{31} & 1-\lambda_{31} & 0 & 0 & 0& 0 \\
\lambda_{41} & 0 & 1-\lambda_{41} & 0 & 0& 0 \\
\lambda_{51} & 0 & 0 & 1-\lambda_{51} & 0& 0 \\
\lambda_{61} & 0 & 0 & 0 & 1-\lambda_{61}& 0 \\
\end{array}
\right)\,,\qquad
\Gamma=\left(
\begin{array}{cccccc}
 0 & 0 & 0 & 0 & 0& 0 \\
\gamma_{21} & 0 & 0 & 0 & 0 & 0\\
 0 & \gamma_{32} & 0 & 0 & 0 & 0\\
 0 & 0 & \gamma_{43} & 0 & 0& 0 \\
 0 & 0 & 0 & \gamma_{54} & 0 & 0\\
 0 & 0 & 0 & 0 & \gamma_{65} & 0\\
\end{array}
\right)\, . \label{SO}
\end{align}
Our goal is to construct a 5-stage third order $2N^*$ low-storage scheme with the largest possible SSP coefficient. 
We will use an additional fifth  stage to improve, not only the SSP coefficient, but also other relevant properties.

The Butcher tableau for the 9-parameter RK methods \eqref{SO} is given by
\begin{align}\label{RKq53new}
\hbox{ 
\tabcolsep 4pt
\begin{tabular}{c|ccccc}
     &  $0$ & $0$&  $0$&  $0$&  $0$     \\ [0.25ex]
    & $\gamma_{21}$ & $0$ &  $0$ &  $0$&  $0$ \\ [0.5ex]
     & $(1-\lambda_{31}) $ & $\gamma_{32}$  & $0$& $0$&  $0$   \\[0.5ex]
     & $(1-\lambda_{41})(1-\lambda_{31})\gamma_{21}$ & $(1-\lambda_{41})\gamma_{32}$& $\gamma_{43}$  & $0$ & $0$   \\[0.5ex]
   & $(1-\lambda_{51})(1-\lambda_{41})(1-\lambda_{31})\gamma_{21}$ & $(1-\lambda_{51})(1-\lambda_{41})\gamma_{32}$& $(1-\lambda_{51})\gamma_{43}$  
 & $\gamma_{54}$ & $0$   \\[0.5ex]
\hline
    & ${b_1}$ & ${b_2}$& ${b_3}$  
 & ${b_4}$ & ${b_5}$   \\[0.5ex]
  \end{tabular}      
 }  
\end{align}
where
\begin{align*}
        {b_1}&= (1-\lambda_{61})(1-\lambda_{51})(1-\lambda_{41})(1-\lambda_{31})\gamma_{21}\,, 
        			&  {b_4}& = \gamma_{54}(1 -\lambda_{61})\, ,\\[0.5ex]
        {b_2}&=(1-\lambda_{61})(1-\lambda_{51})(1-\lambda_{41})\gamma_{32}\,, 
        			&  {b_5} &= \gamma_{65}\, . \\[0.5ex]
         {b_3}&= (1-\lambda_{61})(1-\lambda_{51})\gamma_{43}\,,  & \\
 \end{align*}         
If we define the new parameters
\begin{align*}
   {u}&=\frac{1}{(1-\lambda_{61})(1-\lambda_{51})(1-\lambda_{41})(1-\lambda_{31})}\,, &   {v}&= \frac{1}{(1-\lambda_{61})(1-\lambda_{51})(1-\lambda_{41})}\,,
   \\[2ex]  {w}&= \frac{1}{(1-\lambda_{61})(1-\lambda_{51})}\, , &    {x}& = \frac{1}{(1-\lambda_{61})}\, , 
\end{align*}
then the Butcher tableau for method \eqref{RKq53new} can be written as
\begin{equation} \label{RKq53new2}
\hbox{ 
\tabcolsep 4pt
\begin{tabular}{c|ccccc}
   $0$  &  $0$ & $0$&  $0$&  $0$&  $0$     \\ [0.25ex]
   $ ub_1 $  & $ub_1$ & $0$& $0$   & $0$ & $0$   \\[0.5ex]
    $ v(b_1+b_2) $  & $vb_1$ & $vb_2$& $0$   & $0$ & $0$   \\[0.5ex]
   $w(b_1+b_2+b_3)$  & $wb_1$ & $wb_2$& $wb_3$   & $0$ & $0$    \\[0.5ex]
 $x(b_1+b_2+b_3+b_4)$  & $xb_1$ & $xb_2$& $xb_3$   & $xb_4$ & $0$   \\[0.5ex]
\hline
    & $b_1$ & $b_2$& $b_3$   & $b_4$ & $b_5$   \\[0.5ex]
  \end{tabular}      
 }  
\end{equation}
and the Shu-Osher representation \eqref{SO} is given by
\begin{align} \label{SO_rep}
\Lambda=
\left(
\begin{array}{cccccc}
 0 & 0 & 0 & 0 & 0& 0 \\[0.5ex]
 1 & 0 & 0 & 0 & 0& 0 \\[0.5ex]
 \frac{u-v}{u} &  \frac{v}{u} & 0 & 0 & 0& 0 \\[0.5ex]
 \frac{v-w}{v}  & 0 & \frac{w}{v}& 0 & 0& 0 \\[0.5ex]
 \frac{w-x}{w}  & 0 & 0 & \frac{x}{w}& 0& 0 \\[0.5ex]
 \frac{x-1}{x} & 0 & 0 & 0 & \frac{1}{x}& 0  
\end{array}
\right)\,,\qquad
\Gamma=\left(
\begin{array}{cccccc}
 0 & 0 & 0 & 0 & 0& 0 \\[0.5ex]
 b_1u & 0 & 0 & 0 & 0 & 0\\[0.5ex]
 0 & b_2v & 0 & 0 & 0 & 0\\[0.5ex]
 0 & 0 & b_3w & 0 & 0& 0 \\[0.5ex]
 0 & 0 & 0 & b_4x & 0 & 0\\[0.5ex]
 0 & 0 & 0 & 0 & b_5 & 0 
\end{array}
\right)\, . 
\end{align}
Observe that $\Lambda\geq 0$ if and only if 
\begin{align}\label{eq:uvwx1}
u\geq v \geq w \geq x\geq 1\, .
\end{align}
Consequently,  the maximum value in each column of the Butcher tableau \eqref{RKq53new2} is the subdiagonal one. As
$$R(\mathbb{A})\leq \frac{1}{\max\{a_{ij}, b_j\}}\, ,
$$
From \eqref{eq:uvwx1}
we get that 
$$R(\mathbb{A})\leq \frac{1}{\max\{u b_1, v b_2, w b_3, x b_4, b_5\} }\, . 
$$

\subsection{Construction of 5-stage third order SSP  methods SSP53\_$2N^*_1$ and SSP53\_$2N^*_2$}
  
\noindent{\bf Method  SSP53\_$2N^*_1$:} 
Having in mind that methods \eqref{SO_rep} can be implemented in $2N^*$ memory registers, we look for the optimal third order SSP method 
in this family. We have used standard numerical optimization techniques to get the 9 unknowns in \eqref{SO_rep}, namely $b_i, i=1,\ldots,5; u,v,w,x$. More precisely, we have solved the following  optimization problem.
\begin{equation}
\begin{split}\label{codigo}
	& \text{Maximize  $r$ subject to:}\\[1ex]
	&b_i\ge 0\,, \ i=1,\ldots,5\,, \\
	&u\ge v \ge w \ge x \ge 1\,, \\
	&\text{Third order conditions \eqref{cod_o30}-\eqref{cod_o3}}  \,.\end{split}
\end{equation}
The optimum SSP coefficient $r=2.18075$   is obtained when  $u=v=w=2.33320$ and $x=1$. 
{This means that the method consists in three FE-steps, and average for the fifth stage and a last FE-step to get the numerical solution.} 
The Butcher coefficients for this method are  
\begin{align} \label{RKq53new2180749}
\hbox{ 
\tabcolsep 5pt
\begin{tabular}{c|ccccc}
  $0$  &  $0$ & $0$&  $0$&  $0$&  $0$     \\ [0.25ex]
  $0.443568$  & $0.443568$ & $0$ &  $0$ &  $0$&  $0$ \\ [0.5ex]
  $0.73468$  & $0.443568$ & $0.291111$  & $0$& $0$&  $0$   \\[0.5ex]
  $1.00529$  & $0.443568$ & $0.291111$& $0.270613$  & $0$ & $0$   \\[0.5ex]
 $0.541442$  & $0.190112$ & $0.124769$& $0.115984$& $0.110578$ & $0$   \\[0.5ex]
\hline
    & $0.190112$ & $0.124769$& $0.115984$   & $0.110578$ & $0.458558$   \\[0.5ex]
  \end{tabular}}  
\end{align}
The expanded coefficients  and the Shu-Osher matrices for this method are given in the Appendix (see \eqref{5eo3SSP_opt1}).
  Observe that the largest value in the Butcher tableau is $b_5=1/r$. 
For this method, the 2-norm of the coefficients in the leading term of the local error   is $0.027841$, and the stability function   is given by
\begin{align}
R(z)  = 1+z+\frac{1}{2}z^2+\frac{1}{6}z^3+0.027360 z^4+0.001772\,z^5\, . 
\end{align}
The stability region for this method (dotted line  in left Figure \ref{Fig_stabR}) is  larger than the one  for  the 4-stage third order optimal SSP methods (dashed line).

\begin{figure}[ht]\label{Fig_stabR}
\begin{center} 
{\includegraphics[scale=0.25]{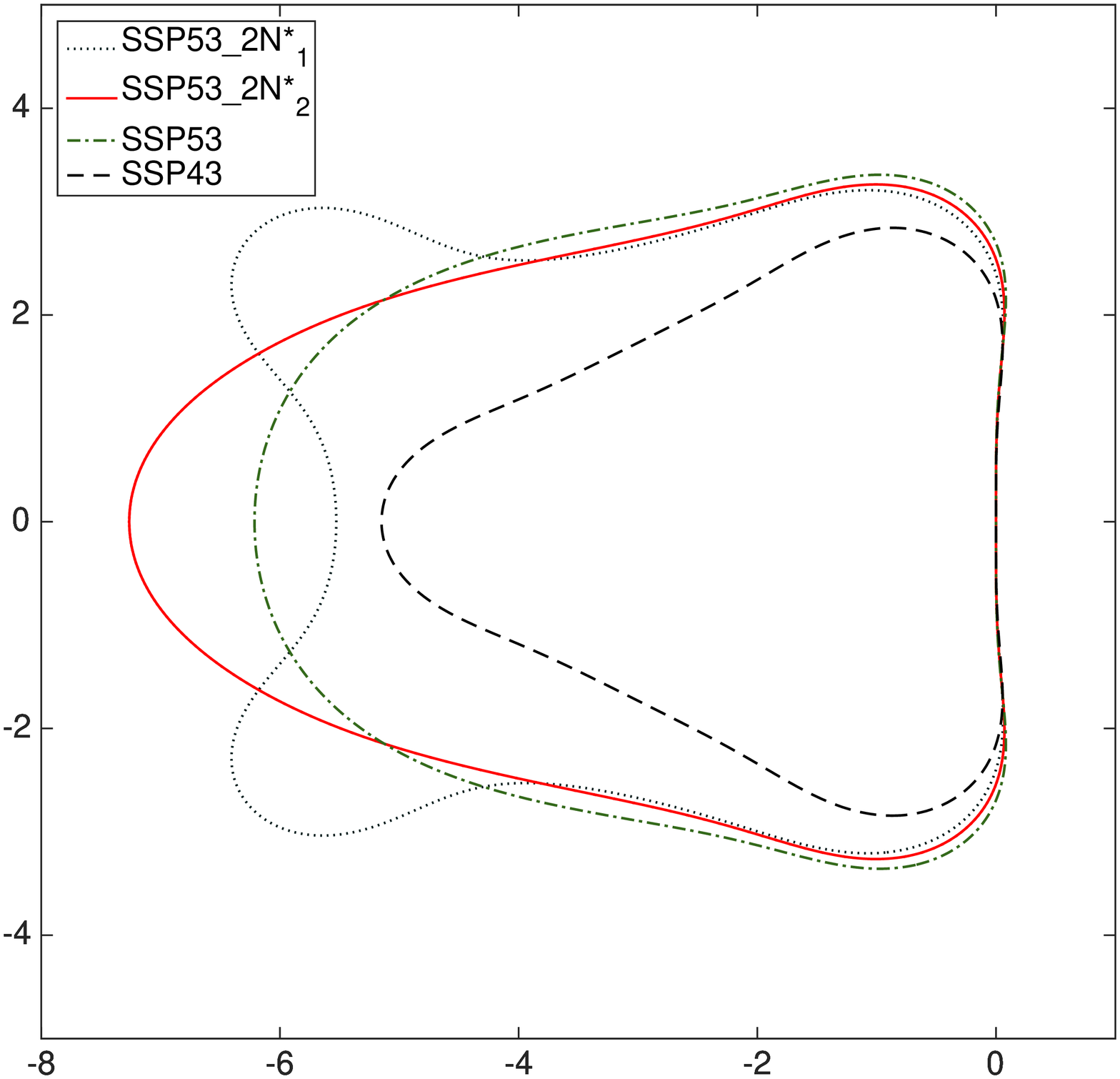}}
{\includegraphics[scale=0.25]{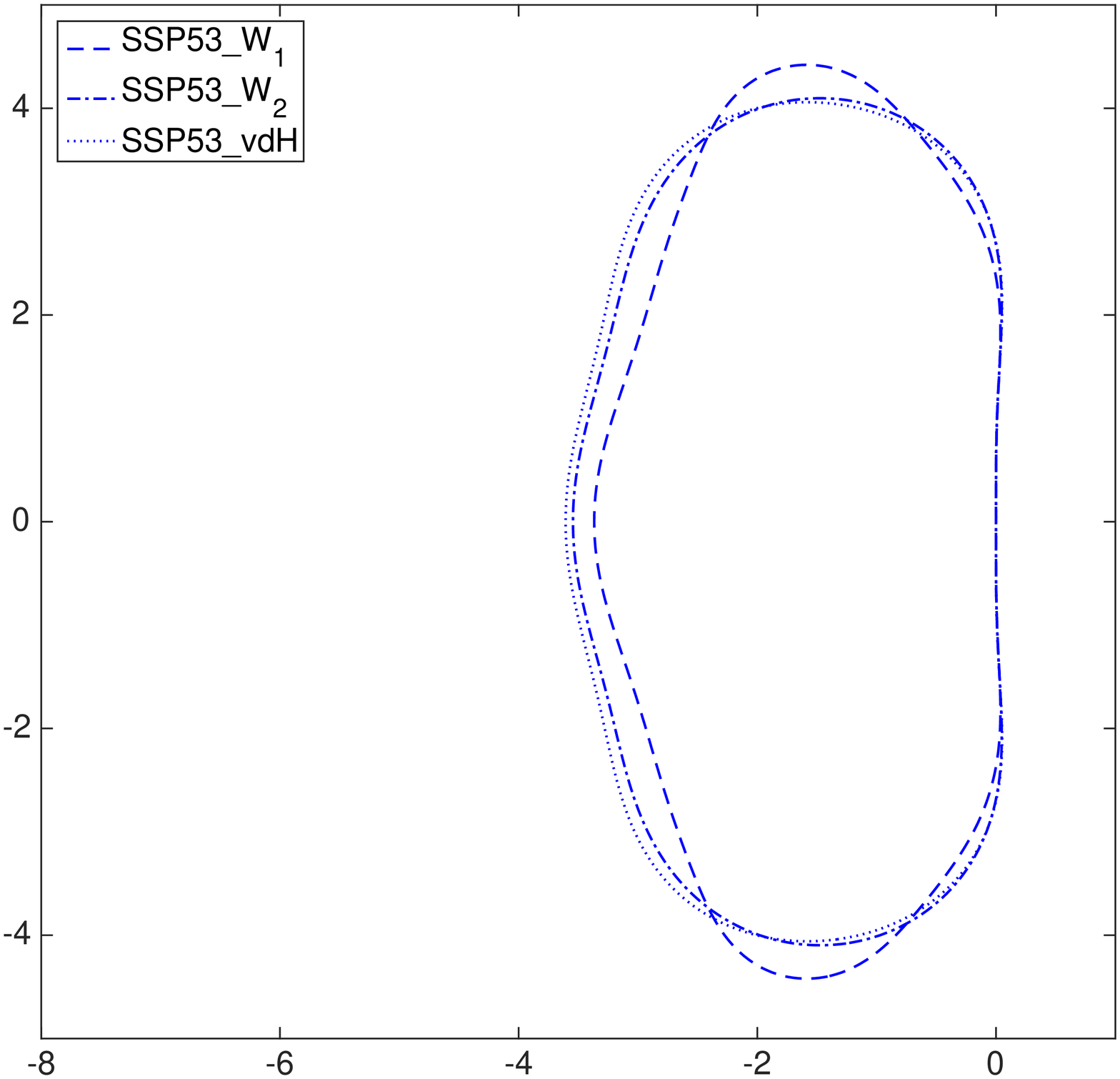}}
\caption{Stability regions for schemes SSP53\_$2N^*_1$ \eqref{RKq53new2180749}, SSP53\_$2N^*_2$ \eqref{RKq53new214}, SSP53 and SSP43 (left).  Stability regions for schemes
SSP53\_W$_1$, SSP53\_W$_2$ and SSP53\_vdH (right).}
\end{center}
\end{figure}

\medskip

\noindent{\bf Method  SSP53\_$2N^*_2$:} 
Motivated by the $2N^*$ structure of the optimum 4-stage third order SSP method \eqref{2N3ordern2}, we have considered an additional stage preserving this low-storage pattern. This is equivalent to consider the case $u=v$ and 
$w=x$ in $2N^*$ methods~\eqref{SO_rep}.
\begin{align} \label{SO_rep2}
\Lambda=
\left(
\begin{array}{cccccc}
 0 & 0 & 0 & 0 & 0& 0 \\
 1 & 0 & 0 & 0 & 0& 0 \\
 0 & 1 & 0 & 0 & 0& 0 \\
 \frac{v-x}{v} & 0 & \frac{x}{v} & 0 & 0& 0 \\
 0 & 0 & 0 & 1 & 0& 0 \\
 \frac{x-1}{x} & 0 & 0 & 0 & \frac{1}{x}& 0 \\
\end{array}
\right)\,,\qquad
\Gamma=\left(
\begin{array}{cccccc}
 0 & 0 & 0 & 0 & 0& 0 \\
 b_1v & 0 & 0 & 0 & 0 & 0\\
 0 & b_1v & 0 & 0 & 0 & 0\\
 0 & 0 & b_3x & 0 & 0& 0 \\
 0 & 0 & 0 & b_4x & 0 & 0\\
 0 & 0 & 0 & 0 & b_5 & 0\\
\end{array}
\right)\, . 
\end{align}
These methods are a subclass of methods \eqref{SO_rep} and consequently can be implemented in $2N^*$ memory registers. We have used numerical optimization techniques analogous  to \eqref{codigo}
 to get the 6 unknowns in \eqref{SO_rep2}. Proceeding in this way, we have obtained numerically
 the optimum SSP coefficient $r=2.14874$ when $a_{21}=a_{54}=0.465389$. The Butcher coefficients for
this method are
\begin{align}\label{RKq53new214}
\hbox{ 
\tabcolsep 5pt
\begin{tabular}{c|ccccc}
0 & 0 & 0 & 0 & 0 & 0 \\
0.465389 & 0.465389 & 0 & 0 & 0 & 0 \\
0.930778 & 0.465389 & 0.465389 & 0 & 0 & 0 \\
0.420414 & 0.147834 & 0.147834 & 0.124746 & 0 & 0 \\
0.885802  &  0.147834  & 0.147834  & 0.124746 & 0.465389 & 0  
 \\[0.5ex] 
\hline
     & 0.141147  & 0.141147  & 0.119103  & 0.444339 & 0.154263 
\end{tabular}      
}  
\end{align} 
The expanded coefficients for this method are given in the Appendix (see \eqref{5eo3SSP_opt2}). 
Observe that the largest value in the Butcher tableau is $a_{21}=a_{32}=a_{54}=1/r$. The 2-norm of  the coefficients in the leading term of the local error  is $0.022736$, and the stability function   is 
\[
R(z)  = 1+z+\frac{1}{2}z^2+\frac{1}{6}z^3+0.027360\, z^4 + 0.001772\, z^5\, . 
\] 
The stability region is the largest one in Figure \ref{Fig_stabR} (continuous line). The left hand side of the absolute stability interval for this method is $-7.26$. 

For completeness, in Figure \ref{Fig_stabR} (right) we show the stability regions of low-storage methods based on van der Houven and Williamson techniques. Observe that for these schemes the stability intervals are smaller than the ones in Figure \ref{Fig_stabR} (left), related to SSP(5,3) low-storage methods based on Shu-Osher matrices.

\section{Numerical experiments}\label{sec:numexp}
In this section we study the performance of the new 5-stage  third order low-storage SSP  RK methods, namely  SSP53\_$2N^*_1$ \eqref{RKq53new2180749}  and  
SSP53\_$2N^*_2$ \eqref{RKq53new214}. Our goal is to study the effective SSP coefficient when a given problem is integrated. For this purpose, we have considered the hyperbolic   Buckley-Leverett equation \eqref{BL} whose solution is Total Variation Diminishing (TVD). 

In order to compare the behaviour of the new methods,   we have also considered other 5-stage  third order low-storage SSP  RK methods from the literature. Specifically we have dealt with four optimal SSP(5,3) methods, namely  the schemes SSP53\_R  \eqref{method_Ruuth}, SSP53\_H \eqref{RKq53new2H}, SSP53$_1$ \eqref{5eo3SSP_opt1} and  SSP53$_2$  \eqref{5eo3SSP_opt2}.
The first three schemes  are $3N$ low-storage methods, while more than three memory registers are needed to implement the fourth one. Besides, we have also considered two Williamson schemes, namely the methods
SSP53\_W$_1$ \eqref{Williamson1} and 
SSP53\_W$_2$    \eqref{Williamson2}, and the van der Houwen method    SSP53\_vdH  \eqref{vdHouwen}. 
Finally, we have also considered the method SSP43 \eqref{2N3ordern2}, the  optimum 4-stage  third order  SSP  RK method. References, expanded coefficients and more details about  these methods can be seen in the Appendix section \ref{appendix}.

 \begin{table}[h]
\begin{center}
    \begin{tabular}{ |lc|c|c|c|c|c|c | }
    \hline
    & &  {\bf Stages} &  {\bf Order} &  {\bf SSP} & {\bf Observed SSP} & {\bf Error} & {\bf Number of}  \\
    & & & &   {\bf coefficient}  & {\bf  coefficient} & {\bf  constant} & {\bf registers}   \\ \hline\hline 
   SSP53\_$2N{^*_1}$ & \eqref{RKq53new2180749} &$5$ & $3$ & $2.1807$ & $ 2.29$   &2.78407e-02&$2N^*$       \\     \hline
      SSP53\_$2N{^*_2}$ & \eqref{RKq53new214} &$5$ & $3$ & $2.1487$ & $ 2.45$   &2.27362e-02&$2N^*$    \\\hline            
\hline
SSP53$_1$ & \eqref{5eo3SSP_opt1}  &   $5$ & $3$ & $2.6506$ &$2.96$ &1.48757e-02& $3N$     
 \\   \hline
  SSP53$\_{\text R}$ & \eqref{method_Ruuth} &$5$ & $3$ & $2.6506$ & $2.90$  &1.66219e-02& $3N$    
  \\  \hline
SSP53$_2$  &\eqref{5eo3SSP_opt2}  &$5$ & $3$ & $2.6506$ & $2.78$  &1.81787e-02& $\ge 3N$    
 \\  \hline
  SSP53$\_{\text H}$ &\eqref{RKq53new2H} &$5$ & $3$ & $2.6506$ & $ 2.72$   &1.98589e-02&$3N$     \\          \hline\hline
  SSP43 & \eqref{2N3ordern2}&$4$ & $3$ & $2$ & $ 2.04$   &3.60844e-02&$2N^*$      \\       \hline   \hline
    SSP53\_W$_1$  & \eqref{Williamson1} &$5$ & $3$ & $1$ & $ 2.04$   &2.14944e-02&$2N$-W         \\          \hline
    SSP53\_W$_2$  & \eqref{Williamson2} &$5$ & $3$ & $1.4015$ & $2.20$   &2.88494e-02&$2N$-W         \\          \hline \hline
    SSP53\_vdH  & \eqref{vdHouwen} &$5$ & $3$ & $1.4828$ & $1.96$   &2.55799e-02&$2N$-vdH         \\          \hline
 \end{tabular}
\end{center}
   \caption{Theoretical and observed SSP coefficients, error constant and number of memory registers. Top: new SSP(5,3) $2N^*$ low-storage schemes. Middle: optimal SSP(5,3) schemes. Next-to-last: optimal SSP(4,3) 
   $2N^*$ low-storage scheme. 
   End: optimal SSP 5-stage third order $2N$ Williamson and van der Houwen low-storage schemes.} 
     \label{tabla:observed}
    \end{table}

\subsection{Hyperbolic 1-dimensional Buckley-Leverett equation}
The  hyperbolic 1-dimensional Buckley-Leverett equation is defined by  (see, e.g., \cite{FeSpi_aplnum,LeVeque2002}) 
\begin{equation}\label{BL}
    \frac{\partial}{\partial t}u(x,t)+\frac{\partial}{\partial x}\Phi(u(x,t))=0 \, , \qquad \text{with} \qquad 
    \Phi(u)=\frac{3u^2}{(1-v)^2}\, .
\end{equation} 
We consider $0\le x\le 1$, $0\le t\le 1/8$, periodic boundary condition $u(0,t)=u(1,t)$ and initial condition
\[
     u(x,0)=\begin{cases} 
           0 & \text{for } 0<x\le 1/2, \\[0.5ex]
           \tfrac{1}{2} & \text{for }  \tfrac{1}{2}<x\le 1\,.
\end{cases}
\]
We semi-discretize this problem using a uniform grid with mesh-points 
$x_j=j\Delta x$
where $j = 1, 2, . . . , N$ and
$x = 1/N$, $N = 100$. We denote $U_j(t)\approx u(x_j,t)$ and we approximate \eqref{BL} by the system of ordinary differential equations
\begin{align}
U_j'(t)=\frac{1}{\Delta x}(\Phi(U_{j-1/2}(t))-\Phi(U_{j+1/2}(t))), \qquad j = 1, 2, \ldots, N, \label{ode_LB}
\end{align}
where 
$$U_{j+1/2}(t)= U_j+\frac{1}{2} \phi(\theta_j)\left(U_{j+1}-U_j\right)\, , $$
and $\phi(\theta)$ is the Koren's limiter defined by 
$$\phi(\theta)=\max\left(0, \min\left(2, \frac{2}{3}+\frac{1}{3}\theta, 2\theta\right)\right)\, , \quad \hbox{where}\quad \theta_j=\frac{U_j-U_{j-1}}{U_{j+1}-U_j}\, . 
$$
In order to  compute the observed SSP coefficient for a given explicit RK, we have integrated  \eqref{ode_LB} 
with different stepsizes, from 
$\Delta t=2\cdot 10^{-3}\ $ to $\ \Delta t=10^{-2}$. For each step size $\Delta t$, the maximal ratio of the TV-seminorm of two consecutive numerical approximations, in the time interval $[0, 1/8]$, is computed
\[
\mu(\Delta t)=\max\left\{  \frac{\|u_n\|_{TV}}{\|u_{n-1}\|_{TV}}
\mid  n\ge 1\,,\text{with } n\Delta t\le 1/8 \right\}\,.
 \]
If $\mu(\Delta t)=1$, then the explicit RK method is Total Variation Diminishing (TVD) on the interval $[0, 1/8]$,  that is ${\|u_n\|_{TV}}\le {\|u_{n-1}\|_{TV}}$ (see \cite{FeSpi_aplnum} for details). 

We have obtained that the forward Euler method is TVD for $0\leq \Delta t\leq \Delta t_{FE}^{obs} 
\simeq 0.0025$.
For a given scheme $\mathbb{A}$,  we have repeated  this computation to obtain the value $\Delta t^{obs}_{\mathbb{A}}$ such that  $\mu(\Delta t^{obs}_{\mathbb{A}})=1$; then the quotient  $\Delta t^{obs}_{\mathbb{A}} /\Delta t_{FE}^{obs}$ gives the observed SSP coefficient of  scheme $\mathbb{A}$, that we will denote by $c_\mathbb{A}^{obs}$. 

In Table \ref{tabla:observed} we have summarized some information on the schemes considered and the numerical results obtained. 
More precisely, for each scheme, we give the number of stages $s$, the order $p$, the theoretical and observed SSP coefficients, the $\|\cdot\|_2$-error constant obtained from the residuals of  the $p+1$ order conditions and, finally, the number of memory registers needed for the implementation.
\begin{figure}[h]
\centering
\includegraphics[viewport=3cm 2.5cm 38cm 14.5cm, clip,scale=0.3]{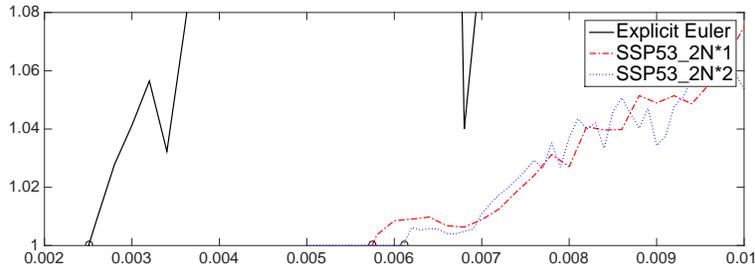} 
\caption{Ratio $\mu(\Delta t)$ for the forward Euler method and the new SSP(5,3) explicit Runge-Kutta methods ($2N^*$ low-storage schemes).} 
   \label{Fig:new_LSmethods}
\end{figure}

In Figure    \ref{Fig:new_LSmethods}  we show $\Delta t_{FE}^{obs}$ and $\Delta t^{obs}_{\mathbb{A}}$ for the new    5-stage third order SSP explicit RK methods SSP53\_$2N{^*_1}$ and 
 SSP53\_$2N{^*_2}$. In both cases, the observed SSP 
 coefficient $c^{obs}_{\mathbb{A}}$, 2.29 and 2.45,  is better than the theoretical one, 2.18 and 2.15, respectively. This increase is more relevant (14\%) for method SSP53\_$2N{^*_2}$ than for method 
SSP53\_$2N{^*_1}$ (5\%). Observe that the   $\|\cdot\|_2$-error constant obtained for method SSP53\_$2N{^*_2}$ is smaller that the one for method SSP53\_$2N{^*_1}$ (see table \ref{tabla:observed}).
 \begin{figure}[ht]
\begin{center}
\includegraphics[viewport=3cm 2.5cm 38cm 14.5cm, clip,scale=0.3]{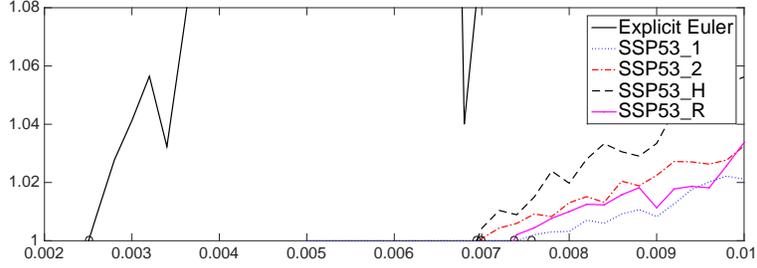}
\end{center}
\caption{Ratio $\mu(\Delta t)$ for the forward Euler method and the four SSP(5,3) methods 
considered:  SSP53$_1$ \eqref{5eo3SSP_opt1}, SSP53$_2$ \eqref{5eo3SSP_opt2}, SSP53\_H \eqref{RKq53new2H} and SSP53\_R \eqref{method_Ruuth}.}
   \label{Fig:SSP53methods}
\end{figure}

In Figure    \ref{Fig:SSP53methods}  we show $\Delta t_{FE}^{obs}$ and $\Delta t^{obs}_{\mathbb{A}}$ for the four  optimal    SSP(5,3)   methods considered in this paper. In all the cases, the observed SSP 
 coefficient $c^{obs}_{\mathbb{A}}$  is better than the theoretical one.
This increase is more relevant (11.6\%) for method SSP53$_1$ than for the other optimal    SSP(5,3)   methods. Observe that the   $\|\cdot\|_2$-error constant obtained for method SSP53$_1$ is lower that the error constant obtained for the other optimal    SSP(5,3)   methods (see table \ref{tabla:observed}).
\begin{figure}[ht]
\begin{center}
\includegraphics[viewport=3cm 2.5cm 38cm 14.5cm, clip,scale=0.3]{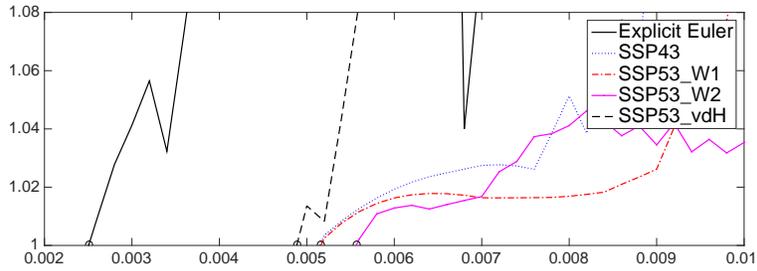}
\end{center}
\caption{Ratio $\mu(\Delta t)$ for the forward Euler method and  other  $2N$ low-storage Runge-Kutta methods:  SSP43 \eqref{2N3ordern2},  SSP53\_W$_1$   \eqref{Williamson1},
SSP53\_W$_2$, \eqref{Williamson2} and SSP53\_vdH   \eqref{vdHouwen}.} 
   \label{Fig:LSmethods}
\end{figure}

Finally, in Figure    \ref{Fig:LSmethods}  we show $\Delta t_{FE}^{obs}$ and $\Delta t^{obs}_{\mathbb{A}}$ for the   the optimal SSP(4,3) method and for the Williamson and van der Houven methods considered in this paper. For method SSP(4,3) the observed and the theoretical SSP coefficient are almost identical. However, for SSP53\_W$_1$, SSP53\_W$_2$ and  SSP53\_vdH methods the observed   coefficient is significantly better than the theoretical one. Despite this increase, the observed SSP coefficient is smaller than the ones for the new $2N^*$ low-storage methods SSP53\_$2N{^*_1}$ and    SSP53\_$2N{^*_2}$ constructed in this paper.

\section{Conclusions} 
In this paper we have studied third order explicit SSP RK methods that can be implemented in $2N^*$ memory registers; besides the SSP coefficient of the schemes, we are interested in some other relevant properties. The optimal SSP(4,3) scheme belongs to this class of methods but, due to its uniqueness, any other relevant properties of the method  are determined.  

We have studied the family of optimum SSP(5,3) methods and we have proven that they cannot be implemented in $2N$ memory registers. Proposition \ref{prop_3N} shows us that  these methods can be implemented in $3N$ memory registers  just in same cases.

Next, we have constructed new  SSP(5,3)   $2N^*$ low-storage   explicit RK schemes. We have exploited the sparse structure of the Shu-Osher matrices to get SSP(5,3)  methods that  can be implemented with $2N$ memory registers, even 
if we have to retain the previous time step approximation. 

Finally, we have tested the performance of the methods for the Buckley-Leverett equation.  With regard to the difference between the theoretical and observed SSP coefficients, the numerical experiments done show that:
\begin{itemize}
    \item For the optimal SSP(4,3) method, the theoretical stepsize restriction for the Buckley-Leverett equation is sharp. However, for the rest of the methods tested, the observed stepsize restrictions are larger than the ones ensured by the SSP theory. 
    \item For the four optimal SSP(5,3) schemes considered in the paper, the method SSP53$_1$ \eqref{5eo3SSP_opt1}, with the largest observed SSP coefficient, is also the one with smallest $\|\cdot\|_2$-error constant.  
    Besides, the observed SSP coefficients for the new schemes SSP53\_$2N{^*_1}$   \eqref{RKq53new2180749} and SSP53\_$2N{^*_2}$   \eqref{RKq53new214}   are a $5\%$ and $14\%$ larger, respectively, than the theoretical ones. It turns out that the $\|\cdot\|_2$-error constant is smaller for scheme SSP53\_$2N{^*_2}$   \eqref{RKq53new214}. These facts lead us to conjecture that the smaller $\|\cdot\|_2$-error constant is, the larger the observed SSP coefficient.
\end{itemize}

Even though the new methods do not achieve the optimum SSP(5,3) coefficient, namely $r=2.65$, the numerical experiments show that they have good observed SSP coefficient~$c^{obs}_{\mathbb{A}}$ for the problem tested. Besides, they have other relevant properties, as large stability region,  improving these features with respect to other $2N^*$ low-storage SSP RK methods. 
Furthermore, scheme SSP53\_$2N{^*_2}$   \eqref{RKq53new214}, the one with larger stability interval and smaller $\|\cdot\|_2$-error constant, gives slighter better results than method SSP53\_$2N{^*_1}$   \eqref{RKq53new2180749}. This fact shows the relevance of these additional properties in the performance of SSP methods.

\section{Appendix}\label{appendix}
In this section we show the coefficients of the methods considered in the numerical experiments: The four SSP53 optimal methods,  the two $2N^*$ low-storage schemes  obtained in this paper, and finally the  Williamson and van der Houwen low-storage type methods.

\subsection{Optimal SSP53  schemes}
First we give the coefficients of  four optimal  SSP(5,3) methods considered in this paper, namely   SSP53\_R, SSP53\_H, SSP53$_1$ and SSP53$_2$.  For all them, the stability function is given in \eqref{stab_funct} and the SSP coefficient is approximately $r=2.65$.

For each method  we show the Butcher coefficients and below the  Shu-Osher form $(\Lambda, \Gamma)$ such that 
$ \Lambda e = (1,0,0,0,0,0)^t$ and matrix $\Gamma$ is subdiagonal. To get this subdiagonal structure we can  use transformation \eqref{inv} or the expressions in (\ref{SO_opt})-(\ref{eq:l3}). 
In the  first three methods the sparsity of matrix $ \Lambda$ allows a  $3N$  implementation. However, for the last one, $4N$ memory registers are needed.

\subsubsection*{\textbf{Scheme SSP53\_R}}
This method was numerically obtained in  \cite{ruuth2006global}. The 2-norm of the coefficients in the leading term of the local error  is $0.0166219$. 

\begin{align} \label{method_Ruuth}
\hbox{ \footnotesize 
\tabcolsep 5pt
\begin{tabular}{c|ccccc}
0 & 0  & 0  & 0  & 0  & 0    \\ [0.25ex]
0.377268915331368 & 0.377268915331368 & 0  & 0   & 0  & 0     \\[0.5ex]
0.754537830662736 &  0.377268915331368 & 0.377268915331368 & 0  & 0  & 0    \\[0.5ex]
0.728985661612186 &  0.242995220537395  & 0.242995220537395  & 0.242995220537395 & 0  & 0      \\[0.5ex]
0.699226135931669 &  0.153589067695126  & 0.153589067695126  & 0.153589067695126 & 0.23845893284629 & 0    \\[0.5ex]
\hline
&  0.206734020864804  & 0.206734020864804  & 0.117097251841844 & 0.18180256012014 & 0.287632146308408  
  \end{tabular}      
} 
\end{align}

\begin{equation*}
\Lambda =\left( 
\begin{array}{*{5}{c@{\hspace{.25cm}}}c}   
0 & 0 & 0 & 0 & 0 & 0 \\
1& 0 & 0 & 0 & 0 & 0 \\
0& 1& 0 & 0 & 0 & 0 \\
\lambda_{41} & 0& \lambda_{43} & 0 & 0 & 0 \\
\lambda_{51} & 0& 0& \lambda_{54} & 0 & 0 \\
0& 0& \lambda_{63} & 0& \lambda_{65} & 0
\end{array}
\right)\, , \qquad
  \Gamma  =\left( 
\begin{array}{*{5}{c@{\hspace{.25cm}}}c}
0 & 0 & 0 & 0 & 0 & 0 \\
\gamma_{21} & 0 & 0 & 0 & 0 & 0 \\
0& \gamma_{32} & 0 & 0 & 0 & 0 \\
0& 0& \gamma_{43} & 0 & 0 & 0 \\ 
0& 0& 0& \gamma_{54} & 0 & 0 \\
0& 0& 0& 0& \gamma_{65} & 0
\end{array}
\right)\,.
\end{equation*}

\medskip\noindent
{\small $\lambda_{41}=0.355909775063327$\,, 
$\lambda_{43}=0.644090224936674$\,, \\
$\lambda_{51}=0.367933791638137$\,, 
$\lambda_{54}=0.632066208361863$\,, \\
$\lambda_{63}=0.237593836598569$\,, 
$\lambda_{65}=0.762406163401431$\,;\\
$\gamma_{21}=0.377268915331368$\,, 
$\gamma_{32}= 0.377268915331368$\,, 
$\gamma_{43}=0.242995220537396$\,, 
$\gamma_{54}= 0.238458932846290$\,, \\
$\gamma_{65}= 0.287632146308408$\,.}

\subsubsection*{\textbf{Scheme SSP53\_H}}
The 2-norm of the coefficients in the leading term of the local error  is  $0.019859$. 
\begin{align} \label{RKq53new2H}
\hbox{  \footnotesize
\tabcolsep 5pt
\begin{tabular}{c|ccccc}
  0  &  0 & 0&  0&  0&  0     \\ [0.25ex]
   0.377268915331368   & 0.377268915331368 & 0& 0   & 0 & 0   \\[0.5ex]
    0.754537830662737  & 0.377268915331368 & 0.377268915331368& 0   & 0 & 0   \\[0.5ex]
  0.782435937433493  & 0.260811979144498 & 0.260811979144498& 0.260811979144498   & 0 & 0    \\[0.5ex]
0.622731084668631  & 0.219153436331987 & 0.117097251841844& 0.117097251841844   & 0.169383144652957 & 0   \\[0.5ex]
\hline
   & 0.219153436331987 & 0.117097251841844& 0.117097251841844   & 0.169383144652957 & 0.377268915331368   \\[0.5ex]
 \end{tabular}      
}   
\end{align}
 
\begin{equation*}
\Lambda =\left( 
\begin{array}{*{5}{c@{\hspace{.25cm}}}c}   
0 & 0 & 0 & 0 & 0 & 0 \\
1& 0 & 0 & 0 & 0 & 0 \\
0& 1& 0 & 0 & 0 & 0 \\
\lambda_{41} & 0& \lambda_{43} & 0 & 0 & 0 \\
\lambda_{51} & \lambda_{52}& 0& \lambda_{54} & 0 & 0 \\
0& 0&  0 & 0& 1 & 0
\end{array}
\right)\, , \qquad
  \Gamma  =\left( 
\begin{array}{*{5}{c@{\hspace{.25cm}}}c}
0 & 0 & 0 & 0 & 0 & 0 \\
\gamma_{21} & 0 & 0 & 0 & 0 & 0 \\
0& \gamma_{32} & 0 & 0 & 0 & 0 \\
0& 0& \gamma_{43} & 0 & 0 & 0 \\ 
0& 0& 0& \gamma_{54} & 0 & 0 \\
0& 0& 0& 0& \gamma_{65} & 0
\end{array}
\right)\,.
\end{equation*}

\medskip\noindent
{\small $\lambda_{41}=0.308684154602513$\,, 
$\lambda_{43}=0.691315845397487$\,, \\
$\lambda_{51}=0.280514990468574$\,, 
$\lambda_{52}=0.270513101776498$\,, 
$\lambda_{54}=0.448971907754928$\,;\\
$\gamma_{21}=0.377268915331368$\,, 
$\gamma_{32}= 0.377268915331368$\,, 
$\gamma_{43}=0.260811979144498$\,, 
$\gamma_{54}= 0.169383144652957$\,, \\
$\gamma_{65}= 0.377268915331368$\,. }

\subsubsection*{\textbf{Scheme SSP53$_1$}}
The coefficients of this method have been obtained with    the code RK--Opt \cite{RKOpt}. The 2-norm of the coefficients in the leading term of the local error  is  $0.014876$.  

\begin{align} \label{David1}
\hbox{  \footnotesize
\tabcolsep 5pt
\begin{tabular}{c|ccccc}
  0  &  0 & 0&  0&  0&  0     \\ [0.25ex]
   0.377268915331368   & 0.377268915331368 & 0& 0   & 0 & 0   \\[0.5ex]
    0.754537830662736  & 0.377268915331368 & 0.377268915331368& 0   & 0 & 0   \\[0.5ex]
  0.488281458487577  & 0.162760486162526 & 0.162760486162526& 0.162760486162526   & 0 & 0    \\[0.5ex]
0.788667948667632  & 0.148318743330765 & 0.148299726283723& 0.148299726283723   & 0.343749752769421 & 0   \\[0.5ex]
\hline
   & 0.196490186861586 & 0.117097251841844& 0.117097251841844   & 0.271424313309946 & 0.297890996144780   \\[0.5ex]
 \end{tabular}      
}   
\end{align}

\begin{equation*}
\Lambda =\left( 
\begin{array}{*{5}{c@{\hspace{.25cm}}}c}   
0 & 0 & 0 & 0 & 0 & 0 \\
1& 0 & 0 & 0 & 0 & 0 \\
0& 1& 0 & 0 & 0 & 0 \\
\lambda_{41} & 0& \lambda_{43} & 0 & 0 & 0 \\
\lambda_{51} & \lambda_{52}& 0& \lambda_{54} & 0 & 0 \\
0& \lambda_{62}&  0 & 0& \lambda_{65} & 0
\end{array}
\right)\, , \qquad
  \Gamma  =\left( 
\begin{array}{*{5}{c@{\hspace{.25cm}}}c}
0 & 0 & 0 & 0 & 0 & 0 \\
\gamma_{21} & 0 & 0 & 0 & 0 & 0 \\
0& \gamma_{32} & 0 & 0 & 0 & 0 \\
0& 0& \gamma_{43} & 0 & 0 & 0 \\ 
0& 0& 0& \gamma_{54} & 0 & 0 \\
0& 0& 0& 0& \gamma_{65} & 0
\end{array}
\right)\,.
\end{equation*}

\medskip\noindent
{\small $\lambda_{41}=0.568582304164742$\,, 
$\lambda_{43}=0.431417695835258$\,, \\
$\lambda_{51}=0.088796463619276$\,, 
$\lambda_{52}=0.000050407140024$\,, 
$\lambda_{54}=0.911153129240700$\,,\\
$\lambda_{62}=0.210401429751688$\,,
$\lambda_{65}=0.789598570248313$\,;\\
$\gamma_{21}=0.377268915331368$\,, 
$\gamma_{32}= 0.377268915331368$\,, 
$\gamma_{43}=0.162760486162526$\,, 
$\gamma_{54}= 0.343749752769421$\,, \\
$\gamma_{65}= 0.297890996144780$\,. }

\subsubsection*{\textbf{Scheme SSP53$_2$}}
The coefficients of this method have been obtained with    the code RK--Opt \cite{RKOpt}. The 2-norm of the coefficients in the leading term of the local error  is  $0.018179$. 
\begin{align} \label{David2}
\hbox{ \footnotesize 
\tabcolsep 5pt
\begin{tabular}{c|ccccc}
  0  &  0 & 0&  0&  0&  0     \\ [0.25ex]
   0.377268915331368   & 0.377268915331368 & 0& 0   & 0 & 0   \\[0.5ex]
    0.754537830662737  & 0.377268915331368 & 0.377268915331368& 0   & 0 & 0   \\[0.5ex]
  0.756398701991139  & 0.252132900663713 & 0.252132900663713& 0.252132900663713   & 0 & 0    \\[0.5ex]
0.659994122684924  & 0.188434549340417 & 0.134873511860921& 0.134873511860921   & 0.201812549622665 & 0   \\[0.5ex]
\hline
  & 0.213322822390311 & 0.166821102311173& 0.117097251841844   & 0.175213758594633 & 0.327545064862039   \\[0.5ex]
 \end{tabular}      
}   
\end{align}
\begin{equation*}
\Lambda =\left( 
\begin{array}{*{5}{c@{\hspace{.25cm}}}c}   
0 & 0 & 0 & 0 & 0 & 0 \\
1& 0 & 0 & 0 & 0 & 0 \\
0& 1& 0 & 0 & 0 & 0 \\
\lambda_{41} & 0& \lambda_{43} & 0 & 0 & 0 \\
\lambda_{51} & \lambda_{52}& 0& \lambda_{54} & 0 & 0 \\
0& 0&\lambda_{63} & 0& \lambda_{65} & 0
\end{array}
\right)\, , \qquad
  \Gamma  =\left( 
\begin{array}{*{5}{c@{\hspace{.25cm}}}c}
0 & 0 & 0 & 0 & 0 & 0 \\
\gamma_{21} & 0 & 0 & 0 & 0 & 0 \\
0& \gamma_{32} & 0 & 0 & 0 & 0 \\
0& 0& \gamma_{43} & 0 & 0 & 0 \\ 
0& 0& 0& \gamma_{54} & 0 & 0 \\
0& 0& 0& 0& \gamma_{65} & 0
\end{array}
\right)\,.
\end{equation*}

\medskip\noindent
{\small $\lambda_{41}=0.331689173378475$\,, 
$\lambda_{43}=0.668310826621525$\,, \\
$\lambda_{51}=0.323099315304423$\,, 
$\lambda_{52}=0.141970449466930$\,, 
$\lambda_{54}=0.534930235228647$\,,\\
$\lambda_{63}=0.131799489564770$\,,
$\lambda_{65}=0.868200510435230$\,;\\
$\gamma_{21}=0.377268915331368$\,, 
$\gamma_{32}= 0.377268915331368$\,, 
$\gamma_{43}=0.252132900663713$\,, 
$\gamma_{54}= 0.201812549622665$\,, \\
$\gamma_{65}= 0.327545064862039$\,. }

\subsection{Optimal SSP 5-stage third order $2N^*$  schemes }
In this section we give the coefficients the two optimal   SSP 5-stage third order methods that can be implemented in $2N^*$ memory registers. For each one we show the Butcher coefficients and below the  Shu-Osher form $(\Lambda, \Gamma)$ such that $ \Lambda e = (1,0,0,0,0,0)^t$ and matrix $\Gamma$ is subdiagonal.  In both cases all the entries are in the first subdiagonal and the first column of $\Lambda$.

\subsubsection*{\textbf{Scheme SSP53\_$2N{^*_1}$}}
This is the optimum method of the family \eqref{SO}. The coefficient SSP is $r=2.180749177932739$ and the stability function is
\[
R(z)  = 1+z+\frac{1}{2}z^2+\frac{1}{6}z^3+0.027360346839505386\, z^4 + 
0.0017718595675709542\, z^5\, . 
\]
The coefficients of this method have been obtained by solving the optimization 
problem \eqref{codigo}. 
The 2-norm of the coefficients in the leading term of the local error  is $0.027840660448808976$.
\begin{align}\label{5eo3SSP_opt1}
\hbox{ \footnotesize
\tabcolsep 5pt
\begin{tabular}{c|ccccc}
0  &0 & 0 & 0 & 0 & 0 \\
0.443568244942995 & 0.443568244942995 & 0 & 0 & 0 & 0 \\
0.734679665016762 & 0.443568244942995 & 0.291111420073766 & 0 & 0 & 0 \\
1.005292266294979 & 0.443568244942995 & 0.291111420073766 & 0.27061260127822 & 0 & 0 \\
0.541442494648948 & 0.190111792195291 & 0.124769332407581 & 0.11598361065329 & 0.110577759392786 & 0  
\\[0.5ex]
\hline
      & 0.190111792195291 & 0.124769332407581 & 0.11598361065329 & 0.110577759392786 & 0.4585575053510519
     \end{tabular}      
}  
\end{align}

\begin{equation*}
\Lambda =\left( 
\begin{array}{*{5}{c@{\hspace{.25cm}}}c}   
0 & 0 & 0 & 0 & 0 & 0 \\
1& 0 & 0 & 0 & 0 & 0 \\
0& 1& 0 & 0 & 0 & 0 \\
0 & 0& 1 & 0 & 0 & 0 \\
\lambda_{51} & 0& 0& \lambda_{54} & 0 & 0 \\
0& 0& 0 & 0&1 & 0
\end{array}
\right)\, , \qquad
  \Gamma  =\left( 
\begin{array}{*{5}{c@{\hspace{.25cm}}}c}
0 & 0 & 0 & 0 & 0 & 0 \\
\gamma_{21} & 0 & 0 & 0 & 0 & 0 \\
0& \gamma_{32} & 0 & 0 & 0 & 0 \\
0& 0& \gamma_{43} & 0 & 0 & 0 \\ 
0& 0& 0& \gamma_{54} & 0 & 0 \\
0& 0& 0& 0& \gamma_{65} & 0
\end{array}
\right)\,.
\end{equation*}

\medskip\noindent
{\small $\lambda_{51}=0.571403511494104$\,, 
$\lambda_{54}=0.428596488505896$\,;\\
$\gamma_{21}=0.443568244942995$\,, 
$\gamma_{32}= 0.291111420073766$\,, 
$\gamma_{43}=0.270612601278217$\,, 
$\gamma_{54}= 0.110577759392786$\,, \\
$\gamma_{65}= 0.458557505351052$\,. }

\subsubsection*{Scheme SSP53\_$2N{^*_2}$}
This is the optimum method of the family \eqref{SO_rep2}. The coefficient SSP is 
$r=2.1487419827223833$ and the stability function is
\[
R(z)  = 1+z+\frac{1}{2}z^2+\frac{1}{6}z^3+0.029448369208272717\, z^4 + 
0.0019397052596758003\, z^5\, . 
\]
The 2-norm of the coefficients in the leading term of the local error  is $0.0227362$.
\begin{align}\label{5eo3SSP_opt2}
\hbox{ \footnotesize
\tabcolsep 5pt
\begin{tabular}{c|ccccc}
0 & 0 & 0 & 0 & 0 & 0 \\
0.465388589249323 & 0.465388589249323 & 0 & 0 & 0 & 0 \\
0.930777178498646 & 0.465388589249323 & 0.465388589249323 & 0 & 0 & 0 \\
0.420413812847710 & 0.147834007766856 & 0.147834007766856 & 0.124745797313998 & 0 & 0 \\
0.885802402097033 &  0.147834007766856 & 0.147834007766856 & 0.124745797313998 & 0.465388589249323 & 0  
 \\[0.5ex] 
\hline
     & 0.141147331533922 & 0.141147331533922 & 0.119103423338902 & 0.444338609844587 & 0.154263303748666
\end{tabular}      
}  
\end{align}

\begin{equation*}
\Lambda =\left( 
\begin{array}{*{5}{c@{\hspace{.25cm}}}c}   
0 & 0 & 0 & 0 & 0 & 0 \\
1& 0 & 0 & 0 & 0 & 0 \\
0& 1& 0 & 0 & 0 & 0 \\
\lambda_{41}  & 0& \lambda_{43}  & 0 & 0 & 0 \\
0& 0& 0& 0 & 1 & 0 \\
\lambda_{61}& 0& 0 & 0&\lambda_{65} & 0
\end{array}
\right)\, , \qquad
  \Gamma  =\left( 
\begin{array}{*{5}{c@{\hspace{.25cm}}}c}
0 & 0 & 0 & 0 & 0 & 0 \\
\gamma_{21} & 0 & 0 & 0 & 0 & 0 \\
0& \gamma_{32} & 0 & 0 & 0 & 0 \\
0& 0& \gamma_{43} & 0 & 0 & 0 \\ 
0& 0& 0& \gamma_{54} & 0 & 0 \\
0& 0& 0& 0& \gamma_{65} & 0
\end{array}
\right)\,.
\end{equation*}

\medskip\noindent
{\small $\lambda_{41}=0.682342861037239$\,, 
$\lambda_{43}=0.317657138962761$\,, \\
$\lambda_{61}=0.045230974482400$\,, 
$\lambda_{65}=0.954769025517600$\,;\\
$\gamma_{21}=0.465388589249323$\,, 
$\gamma_{32}= 0.465388589249323$\,, 
$\gamma_{43}=0.124745797313998$\,, 
$\gamma_{54}= 0.465388589249323$\,, \\
$\gamma_{65}= 0.154263303748666$\,. }

\subsection{Williamson and van der Houwen low-storage type methods}
In this paper we have also considered some optimal Williamson and van der Houwen low-storage type methods obtained in \cite{spiteri2002nco} and \cite{ruuth2006global}. These are  $2N$ low-storage methods. Our interest relies  on the five-stage third-order methods.
\subsubsection*{Scheme SSP53\_W${_1}$}
In \cite{spiteri2002nco}  some optimal Williamson low-storage type method are numerically obtained. Here we show the coefficients of the optimal 5-stages third order method. The SSP coefficient  is 
$r=1$ and the stability function is
\[
R(z)  = 1+z+\frac{1}{2}z^2+\frac{1}{6}z^3+0.028737614071812287\, z^4 + 
0.0037729266088722306\, z^5\, . 
\]
The 2-norm of the coefficients in the leading term of the local error  is $0.0214944$.    
and the Butcher coefficients are 
\begin{align} \label{Williamson1}
\hbox{  \footnotesize
\tabcolsep 5pt
\begin{tabular}{c|ccccc}
  $0$  &  $0$ & $0$&  $0$&  $0$&  $0$     \\ [0.25ex]
  $ 0.67892607116139 $  & $0.67892607116139$ & $0$& $0$   & $0$ & $0$   \\[0.5ex]
   $ 0.34677649493991$  & $0.14022991560621$ & $0.20654657933371$& $0$   & $0$ & $0$   \\[0.5ex]
  $0.66673359500982$  & $0.20569370073026$ & $0.18144649137471$& $0.27959340290485$   & $0$ & $0$    \\[0.5ex]
$0.76590087429032$  & $0.16104646283838$ & $0.19856511041100$& $0.08890670263481$   & $0.31738259840613$ & $0$   \\[0.5ex]
\hline
     & $0.19215670424132$ & $0.18663683901393$& $0.22177739201759$   & $0.09623007655432$ & $0.30319904778284$   \\[0.5ex]
 \end{tabular}      
} 
\end{align}

\subsubsection*{Scheme SSP53\_W${_2}$}
Below we show the coefficients of the optimal 5-stage  third order method numerically obtained in \cite{ruuth2006global}. The SSP coefficient  is 
$r=1.40154693827206$ and the stability function is
\[
R(z)  = 1+z+\frac{1}{2}z^2+\frac{1}{6}z^3+0.030867245346137964\, z^4 + 
0.003908575831813585\, z^5\, . 
\]
The 2-norm of the coefficients in the leading term of the local error  is $0.0288494$.    
\begin{align} \label{Williamson2}
\hbox{  \footnotesize
\tabcolsep 5pt
\begin{tabular}{c|ccccc}
  $0$  &  $0$ & $0$&  $0$&  $0$&  $0$     \\ [0.25ex]
  $ 0.713497331193829 $  & $0.713497331193829$ & $0$& $0$   & $0$ & $0$   \\[0.5ex]
   $ 0.133505249805329$  & $0.133505249805329$ & $0.133505249805329$& $0$   & $0$ & $0$   \\[0.5ex]
  $0.980507830804488$  & $0.133505249805329$ & $0.133505249805329$& $0.713497331193829$   & $0$ & $0$    \\[0.5ex]
$0.566169290867790$  & $0.133505249805329$ & $0.133505249805329$& $0.149579395628566$   & $0.149579395628565$ & $0$   \\[0.5ex]
\hline
     & $0.133505249805329$ & $0.133505249805329$& $0.216758180868589$   & $0.131760203399484$ & $0.384471116121269$   \\[0.5ex]
 \end{tabular}      
} 
\end{align}

\subsubsection*{Scheme SSP53\_vdH}
This is the optimal  5-stage  third order van der Houwen low-storage method numerically obtained in \cite{ruuth2006global}. The SSP coefficient  is 
$r=1.482840341885634$ and the stability function is
\[
R(z)  = 1+z+\frac{1}{2}z^2+\frac{1}{6}z^3+0.030977632110278555\, z^4 + 
0.003801134386056876\, z^5\, . 
\]
The 2-norm of the coefficients in the leading term of the local error  is $0.02557995243600524$.    
\begin{align} \label{vdHouwen}
\hbox{  \footnotesize
\tabcolsep 5pt
\begin{tabular}{c|ccccc}
  $0$  &  $0$ & $0$&  $0$&  $0$&  $0$     \\ [0.25ex]
  $ 0.674381436593749 $  & $0.674381436593749$ & $0$& $0$   & $0$ & $0$   \\[0.5ex]
   $ 0.291120326368482$  & $0.174481959220521$ & $0.116638367147961$& $0$   & $0$ & $0$   \\[0.5ex]
  $0.965501762962231$  & $0.174481959220521$ & $0.116638367147961$& $0.674381436593749$   & $0$ & $0$    \\[0.5ex]
$0.617111102246386$  & $0.174481959220521$ & $0.116638367147961$& $0.162995387938952$   & $0.162995387938952$ & $0$   \\[0.5ex]
\hline
     & $0.174481959220521$ & $0.116638367147961$& $0.162995387938952$   & $0.106256369067643$ & $0.439627916624922$   \\[0.5ex]
 \end{tabular}      
} 
\end{align}

\end{document}